\documentclass[a4paper,11pt, reqno]{amsart}
\usepackage{latexsym}
\usepackage{color}
\usepackage[usenames,dvipsnames]{xcolor}
\usepackage{amssymb,amsfonts,amsmath,mathrsfs}
\usepackage{graphicx}
\newtheorem{theoremA}{Theorem A\!\!}
\newtheorem{theoremB}{Theorem B\!\!}
\newtheorem{theoremC}{Theorem C\!\!}
\newtheorem{theoremD}{Theorem D\!\!}
\newtheorem{theorem}{Theorem}[section]
\newtheorem{conjecture}{Conjecture}[section]

\newtheorem{lemma}[theorem]{Lemma}

\newtheorem{remark}[theorem]{Remark}
\theoremstyle{definition}

\begin{document} 

\title[Variance of sums of arithmetic functions over short intervals]{On the variance of sums of arithmetic functions over primes in short intervals and pair correlation for $L$-functions in the Selberg class}

\begin{abstract}
We establish the equivalence of conjectures concerning the pair correlation of zeros of $L$-functions in the Selberg class and the variances of sums of a related class of arithmetic functions over primes in short intervals. This extends the results of Goldston \& Montgomery [\textbf{\ref{GM}}] and Montgomery \& Soundararajan [\textbf{\ref{MS}}] for the Riemann zeta-function to other $L$-functions in the Selberg class.  Our approach is based on the statistics of the zeros because the analogue of the Hardy-Littlewood conjecture for the auto-correlation of the arithmetic functions we consider is not available in general. One of our main findings is that the variances of sums of these arithmetic functions over primes in short intervals have a different form when the degree of the associated $L$-functions is 2 or higher to that which holds when the degree is 1 (e.g. the Riemann zeta-function).  Specifically, when the degree is 2 or higher there are two regimes in which the variances take qualitatively different forms, whilst in the degree-1 case there is a single regime.  
\end{abstract}

\author{H. M. Bui, J. P. Keating and D. J. Smith}
\address{School of Mathematics, University of Bristol\\
Bristol\\ BS8 1TW\\ UK}
\email{hung.bui@bristol.ac.uk}
\email{j.p.keating@bristol.ac.uk}
\email{dale.smith@bristol.ac.uk}

\thanks{ We gratefully acknowledge support from the Leverhulme Trust and under EPSRC Programme Grant EP/K034383/1
LMF: \textit{L}-Functions and Modular Forms.  JPK is also funded by a Royal Society Wolfson Research Merit Award, and a Royal Society Leverhulme Senior Research Fellowship.   We thank Professors Brian Conrey and Zeev Rudnick for helpful comments.}

\allowdisplaybreaks

\maketitle

\section{Introduction}

Let $\Lambda(n)$ denote the von Mangoldt function, defined by
\begin{equation*}
\Lambda(n) = \begin{cases} \log p & \mbox{if }n=p^k \mbox{ for some prime } p \mbox{ and integer } k \ge 1, \\ 0 & \mbox{otherwise.} \end{cases}
\end{equation*}
The prime number theorem implies that  
\begin{equation*}
\psi(x) := \sum_{n \leq x} \Lambda(n)= x+o(x)
\end{equation*}
as $x \to \infty$, and so determines the average of $\Lambda(n)$ over long intervals.
In many problems one needs to understand sums over shorter intervals.  This is more difficult, because the fluctuations in their values can be large.  To this end Goldston and Montgomery [\textbf{\ref{GM}}] initiated the 
study of the variances
\begin{equation} \label{Zeta variance delta}
V(X, \delta) := \int_{1}^{X} \Big(\psi(x+\delta x)-\psi(x) - \delta x\Big)^{2} dx
\end{equation}
and
\begin{equation} \label{Zeta variance h}
\tilde{V}(X, h) := \int_{1}^{X} \Big(\psi(x+h)-\psi(x) - h\Big)^{2} dx.
\end{equation}
For example, they put forward the following conjecture [\textbf{\ref{GM}}]:
\begin{conjecture}[Variance of primes in short intervals]
\label{GMcon}
For any fixed $\varepsilon>0$
\begin{equation*}
\tilde{V}(X, h)\sim hX\big(\log X-\log
h 
\big)
\end{equation*}
uniformly for $1\leq h\leq X^{1-\varepsilon}$.
\end{conjecture}
This conjecture remains open, but its analogue in the function field setting has recently been proved in the limit of large field size [\textbf{\ref{KR}}].

It is natural to try to compute the variances \eqref{Zeta variance delta} and \eqref{Zeta variance h} using the Hardy-Littlewood Conjecture for the auto-correlation of $\Lambda(n)$:
\begin{equation}
\label{HL}
\sum_{n\le X}\Lambda(n)\Lambda(n+k)\sim \mathfrak{S}(k)X
\end{equation}
as $X\rightarrow\infty$, where $\mathfrak{S}(k)$ is the singular series 
\begin{equation*}
\mathfrak{S}(k) =
 \begin{cases} 
2\prod_{p>2}
\left(1-\frac{1}{(p-1)^2}\right)\prod_{\substack{p>2\\ p|k}}\frac{p-1}{p-2} 
& \mbox{\quad if }k \mbox{ is even, }\\ 0 & \mbox{\quad if }k \mbox{ is odd. } \end{cases}
\end{equation*}
Montgomery and Soundararajan  [\textbf{\ref{MS}}] established that \eqref{HL}, subject to an assumption concerning the implicit error term, implies a more precise asymptotic for the variance $\tilde{V}(X, h)$ when $\log X\leq h\leq X^{1/2}$:
\begin{equation}
\label{Zeta variance h LOT}
\tilde{V}(X, h) = hX\big(\log X-\log
h - \gamma_0-\log 2\pi\big)+O_\varepsilon\Big(h^{15/16}X(\log X)^{17/16}+h^2X^{1/2+\varepsilon}\Big),
\end{equation}
where $\gamma_0$ is the Euler-Mascheroni constant.

An alternative approach to computing the variances \eqref{Zeta variance delta} and \eqref{Zeta variance h} is based on the connection with the Riemann zeta-function $\zeta(s)$ via
\begin{equation*}
\frac{ \zeta^\prime(s)}{\zeta(s)}=-\sum_{n=1}^\infty\frac{\Lambda (n)}{n^s}.
\end{equation*}
This links statistical properties of $\Lambda(n)$ to those of the zeros of the Riemann zeta-function.  Specifically, Goldston and Montgomery [\textbf{\ref{GM}}] proved that Conjecture \ref{GMcon} is equivalent to the following conjecture, due to Montgomery [\textbf{\ref{M}}], concerning the pair correlation of the non-trivial zeros $\frac12 +i\gamma$ of the Riemann zeta-function (in writing the zeros in this form one is assuming the Riemann Hypothesis):
\begin{conjecture}[Pair Correlation Conjecture]\label{SPC}
Let 
\[
\mathcal{F}(X,T)=\sum_{0<\gamma,\gamma'\leq T}X^{i(\gamma-\gamma')}w(\gamma-\gamma'),
\]
where $w(u)=\frac{4}{4+u^2}$.  Then 
for any fixed $A\geq1$ we have
\[
\mathcal{F}(X,T)\sim\frac{T\log T}{2\pi}
\]
uniformly for $T\leq X\leq T^{A}$.
\end{conjecture}

The equivalence between Conjecture \ref{GMcon} and Conjecture \ref{SPC} has been investigated further in [\textbf{\ref{C}}, \textbf{\ref{LPZ}}] to include the lower order terms.

We have two main goals in this paper.  The first is to show how the more precise formula \eqref{Zeta variance h LOT} follows from a more accurate expression for the pair correlation of the Riemann zeros proposed by Bogomolny and Keating [\textbf{\ref{BogK}}] (see also [\textbf{\ref{BeK}}]):
\begin{conjecture}\label{RC}
For $h$ a suitable even test function 
\begin{eqnarray*}
&&\!\!\!\!\!\!\!\!\!\!\!\!\sum_{0<\gamma,\gamma'\leq T}h(\gamma-\gamma')=\frac{h(0)}{2\pi}\int_{0}^{T}\log\frac{t}{2\pi}\,dt+\frac{1}{(2\pi)^2}\int_{0}^{T}\int_{-T}^{T}h(\eta)\bigg[\bigg(\log\frac{t}{2\pi}\bigg)^2\\
&&\qquad+2\bigg(\bigg(\frac{\zeta'}{\zeta}\bigg)'(1+i\eta)+\bigg(\frac{t}{2\pi}\bigg)^{-i\eta}A(i\eta)\zeta(1-i\eta)\zeta(1+i\eta)-B(i\eta)\bigg)\bigg]d\eta dt+O_\varepsilon(T^{1/2+\varepsilon}),
\end{eqnarray*}
where
\[
A(r)=\prod_p\frac{(1-\frac{1}{p^{1+r}})(1-\frac{2}{p}+\frac{1}{p^{1+r}})}{(1-\frac{1}{p})^2}
\]
and
\[
B(r)=\sum_p\bigg(\frac{\log p}{p^{1+r}-1}\bigg)^2.
\]
Here the integral is to be regarded as a principal value near $\eta=0$.
\end{conjecture}
This formula was originally obtained in [\textbf{\ref{BogK}}] from the Hardy-Littlewood Conjecture \eqref{HL}.  Importantly for us here, it was shown by Conrey and Snaith [\textbf{\ref{CS}}] to follow from the ratios conjecture for the Riemann zeta-function [\textbf{\ref{CFZ}}], and in the above formulation we use their notation.  It follows from our general results, set out below, that \eqref{Zeta variance h LOT} may be obtained from an analysis based on Conjecture \ref{RC}.   

The second goal of this paper, and in fact our principal goal, is to extend the approach based on formulae like that in Conjecture \ref{RC} to a wider class of sums in which the von Mangoldt function is multiplied by arithmetic functions associated with other $L$-functions in the Selberg class [\textbf{\ref{S}}].  This essentially corresponds to studying the variances of these functions when summed over prime arguments in short intervals.  

Let $\mathcal{S}$ denote the Selberg class $L$-functions. For $F\in\mathcal{S}$ primitive, 
\[
F(s)=\sum_{n=1}^{\infty}\frac{a_F(n)}{n^s},
\]
let $m_F\geq 0$ be the order of the pole at $s=1$,
\[
\frac{F'}{F}(s)=-\sum_{n=1}^{\infty}\frac{\Lambda_F(n)}{n^s}\qquad\textrm{and}\qquad F(s)^{-1}=\sum_{n=1}^{\infty}\frac{\mu_F(n)}{n^s}\qquad(\textrm{Re}(s)>1).
\]
The function $F(s)$ has an Euler product
\begin{equation}\label{Euler}
F(s)=\prod_{p}\textrm{exp}\bigg(\sum_{l=1}^{\infty}\frac{b_F(p^l)}{p^{ls}}\bigg)
\end{equation}
and satisfies a functional equation
\[
\Phi(s)=\varepsilon_F\overline{\Phi}(1-s),
\]
where
\[
\Phi(s)=Q^s\bigg(\prod_{j=1}^{r}\Gamma(\lambda_j s+\mu_j)\bigg) F(s),
\]
with some $Q>0$, $\lambda_j>0$, $\textrm{Re}(\mu_j)\geq0$ and $|\varepsilon_F|=1$. Here $\overline{\Phi}(s)=\overline{\Phi(\overline{s})}$. We will also write the functional equation in the form
\[ 
F(s)=X(s)\overline{F}(1-s),
\]
where
\[
X(s)=\varepsilon_F Q^{1-2s}\prod_{j=1}^{r}\frac{\Gamma\big(\lambda_j(1-s)+\overline{\mu_j}\big)}{\Gamma(\lambda_j s+\mu_j)}.
\]
The two important invariants of $F(s)$ are the degree $d_F$ and the conductor $\mathfrak{q}_F$,
\[
d_F=2\sum_{j=1}^{r}\lambda_j\qquad\textrm{and}\qquad \mathfrak{q}_F=(2\pi)^{d_F}Q^2\prod_{j=1}^{r}\lambda_j^{2\lambda_j}.
\]

For $F\in\mathcal{S}$, it is expected that a generalised prime number theorem of the form
\begin{equation*}
\psi_{F}(x):= \sum_{n \leq x} \Lambda_F(n)= m_{F} x+o(x)
\end{equation*}
holds. In analogy with \eqref{Zeta variance delta} and \eqref{Zeta variance h} we shall consider
\begin{equation*}
V_F(X, \delta) :=  \int_{1}^{X}\Big |\psi_F(x+\delta x)-\psi_F(x) - m_F\delta x\Big|^{2} dx
\end{equation*}
and
\begin{equation*}
\tilde{V}_F(X, h) :=  \int_{1}^{X}\Big |\psi_F(x+h)-\psi_F(x) - m_Fh\Big|^{2} dx.
\end{equation*}
So, for example, when $F$ represents an $L$-function associated with an elliptic curve, $V_F(X, \delta)$ and $\tilde{V}_F(X, h)$ represent the variances of sums over short intervals involving the Fourier coefficients of the associated modular form evaluated at primes and prime powers; and in the case of Ramanujan's $L$-function, they represent the corresponding variances for sums involving the Ramanujan $\tau$-function.

It is important to note that for most $F\in\mathcal{S}$ one does not expect an analogue of the Hardy-Littlewood Conjecture \eqref{HL}; that is, for most $F\in\mathcal{S}$ it is expected that
\begin{equation*}
\sum_{n\le X}\Lambda_F(n)\Lambda_F(n+h)=o(X).
\end{equation*}
This might lead one to anticipate that $V_F(X, \delta)$ and $\tilde{V}_F(X, h)$ typically exhibit different asymptotic behaviour  than in the case when $F$ is the Riemann zeta-function, because \eqref{HL} plays a central role in our understanding
of the variances in that case.  Somewhat surprisingly from this perspective, our results suggest that $V_F(X, \delta)$ and $\tilde{V}_F(X, h)$ have the same general form for all $F\in\mathcal{S}$.  The reason is that they all look essentially the same from the perspective of the statistical distribution of their zeros.  It would be interesting to understand this from the Hardy-Littlewood point of view.  Presumably it is related to a conspiracy amongst the terms that are $o(X)$, unlike in the case of the Riemann zeta-function where they come from the main term.  Drawing attention to this is one of our principal motivations.

The pair correlation of zeros of $F(s)$ is defined in analogy with the expression in Conjecture \ref{SPC} as
\[
\mathcal{F}_F(X,T)=\sum_{-T\leq\gamma_F,\gamma'_F\leq T}X^{i(\gamma_F-\gamma'_F)}w(\gamma_F-\gamma'_F),
\]
where, assuming the Generalized Riemann Hypothesis (GRH), the non-trivial zeros of $F(s)$ are denoted $\frac12 +i \gamma_F$.  Murty and Perelli [\textbf{\ref{MP}}] conjectured that
\begin{equation*}
\mathcal{F}_F(X,T)\sim  \frac{T\log X}{\pi}
\end{equation*}
uniformly for $T^{A_1}\leq X\leq T^{A_2}$ for any fixed $0<A_1\leq A_2\leq d_F$, and
\begin{equation*}
\mathcal{F}_F(X,T)\sim  \frac{d_FT\log T}{\pi}
\end{equation*}
uniformly for $T^{A_1}\leq X\leq T^{A_2}$ for any fixed $d_F\leq A_1\leq A_2<\infty$. 

Our approach to studying the variances $V_F(X, \delta)$ and $\tilde{V}_F(X, h)$ is based on the pair correlation of zeros.  Specifically, our main results are as stated below.  We set out these results in pairs, because, unlike the case of the Riemann zeta-function and other degree-1 $L$-functions, when $d_F\ge 2$ there are two cases to consider: either $T\leq X\leq T^{d_F}$ or $T^{d_F}\leq X$.  In both of these cases, our results then correspond to examining the implication of the pair correlation of zeros for $V_F(X, \delta)$ (Theorems labelled A), the implications in the reverse direction (B), implications of $V_F(X, \delta)$ for $\tilde{V}_F(X, h)$ (C), and in the reverse direction (D).  

\begin{theoremA}\label{maintheo}
Assume GRH. If $d_F<A_1< A_2<\infty$ and 
\begin{equation}\label{250}
\mathcal{F}_F(X,T)=\frac{T}{\pi}\Big(d_F\log \frac{T}{2\pi}+\log\mathfrak{q}_F-d_F\Big)+O\big(T^{1-c}\big)
\end{equation}
uniformly for $T^{A_1}\ll X\ll T^{A_2}$ for some $c>0$, then for any fixed $1/A_2<B_1\leq B_2<1/A_1$ we have
\begin{eqnarray*}
V_F(X, \delta)&=&\frac{1}{2}\delta X^2\Big(d_F \log\frac{1}{\delta}+\log\mathfrak{q}_F+(1-\gamma_0-\log2\pi)d_F\Big)+O\big(\delta^{1+c/2} X^2\big)\\
&&\qquad\qquad+O_\varepsilon\Big(\delta^{1-\varepsilon} X^2(\delta X^{1/A_2})^{1/2}\Big)+O_\varepsilon\Big(\delta^{1-\varepsilon} X^2(\delta X^{1/A_1})^{-2A_1/(4A_1+1)}\Big)\nonumber
\end{eqnarray*}
uniformly for $X^{-B_2}\ll\delta\ll X^{-B_1}$.
\end{theoremA}

\begin{theoremA}
Assume GRH. If $1<A_1< A_2<d_F$ and 
\begin{equation*}
\mathcal{F}_F(X,T)=\frac{T\log X}{\pi}+O\big(T^{1-c}\big)
\end{equation*}
uniformly for $T^{A_1}\ll X\ll T^{A_2}$ for some $c>0$, then for any fixed $1/A_2<B_1\leq B_2<1/A_1$ we have
\begin{eqnarray*}
V_F(X, \delta)&=&\frac{1}{6}\delta X^2\big(3\log X-4\log 2\big)+O\big(\delta^{1+c/2} X^2\big)\\
&&\qquad\qquad+O_\varepsilon\Big(\delta^{1-\varepsilon} X^2(\delta X^{1/A_2})^{1/2}\Big)+O_\varepsilon\Big(\delta^{1-\varepsilon} X^2(\delta X^{1/A_1})^{-2A_1/(4A_1+1)}\Big)\nonumber
\end{eqnarray*}
uniformly for $X^{-B_2}\ll\delta\ll X^{-B_1}$.
\end{theoremA}
 
\begin{theoremB}\label{theoremB1}
Assume GRH. If $0< B_1< B_2<1/d_F$ and
\begin{eqnarray}\label{B1condition}
V_F(X, \delta)=\frac{1}{2}\delta X^2\Big(d_F \log\frac{1}{\delta}+\log\mathfrak{q}_F+(1-\gamma_0-\log2\pi)d_F\Big)+O\big(\delta^{1+c} X^2\big)
\end{eqnarray}
uniformly for $X^{-B_2}\ll\delta\ll X^{-B_1}$ for some $c>0$, then for any fixed $1/B_2<A_1\leq A_2<1/B_1$ we have
\begin{eqnarray*}
\mathcal{F}_F(X,T)&=&\frac{T}{\pi}\Big(d_F\log \frac{T}{2\pi}+\log\mathfrak{q}_F-d_F\Big)+O_\varepsilon\Big(T ^{3/(3+c)+\varepsilon}\Big)\\
&&\qquad\qquad+O_\varepsilon\Big(T^{1+\varepsilon} \big(T/X^{B_2}\big)^{2}\Big)+O_\varepsilon\Big(T^{1+\varepsilon} \big(T/X^{B_1}\big)^{-1/4}\Big)
\end{eqnarray*}
uniformly for $T^{A_1}\ll X\ll T^{A_2}$. 
\end{theoremB}

\begin{theoremB}
Assume GRH. If $1/d_F< B_1< B_2<1$ and
\begin{eqnarray*}
V_F(X, \delta)=\frac{1}{6}\delta X^2\big(3\log X-4\log 2\big)+O\big(\delta^{1+c} X^2\big)\nonumber
\end{eqnarray*}
uniformly for $X^{-B_2}\ll\delta\ll X^{-B_1}$ for some $c>0$, then for any fixed $1/B_2<A_1\leq A_2<1/B_1$ we have
\begin{eqnarray*}
\mathcal{F}_F(X,T)=\frac{T\log X}{\pi}+O_\varepsilon\Big(T ^{3/(3+c)+\varepsilon}\Big)+O_\varepsilon\Big(T^{1+\varepsilon} \big(T/X^{B_2}\big)^{2}\Big)+O_\varepsilon\Big(T^{1+\varepsilon} \big(T/X^{B_1}\big)^{-1/4}\Big)
\end{eqnarray*}
uniformly for $T^{A_1}\ll X\ll T^{A_2}$. 
\end{theoremB}

\begin{theoremC}\label{secondtheorem} 
Assume GRH. If $0<B_1<B_2\leq B_3<1/d_F$ and
\begin{eqnarray}\label{2.1}
V_F(X, \delta)=\frac{1}{2}\delta X^2\Big(d_F \log\frac{1}{\delta}+\log\mathfrak{q}_F+(1-\gamma_0-\log2\pi)d_F\Big)+O\big(\delta^{1+c} X^2\big)
\end{eqnarray}
uniformly for $X^{-B_3}\ll\delta\ll X^{-B_1}$ for some $c>0$, then we have
\begin{eqnarray}\label{2.2}
\tilde{V}_F(X,h)&=&h X\Big(d_F \log\frac{X}{h}+\log\mathfrak{q}_F-(\gamma_0+\log 2\pi)d_F\Big)\nonumber\\
&&\qquad\qquad+O_\varepsilon\big(hX^{1+\varepsilon}(h/X)^{c/3}\big)+O_\varepsilon\Big(hX^{1+\varepsilon}\big(hX^{-(1-B_1)}\big)^{1/3(1-B_1)}\Big)
\end{eqnarray}
uniformly for $X^{1-B_3}\ll h\ll X^{1-B_2}$.
\end{theoremC}

\begin{theoremC}
Assume GRH. If $1/d_F<B_1<B_2\leq B_3<1$ and
\begin{eqnarray*}
V_F(X, \delta)=\frac{1}{6}\delta X^2\big(3\log X-4\log 2\big)+O\big(\delta^{1+c} X^2\big)
\end{eqnarray*}
uniformly for $X^{-B_3}\ll\delta\ll X^{-B_1}$ for some $c>0$, then we have
\begin{eqnarray*}
\tilde{V}_F(X,h)&=&\frac{1}{6}h X\Big(6\log X-\big(3+8\log 2\big)\Big)\\
&&\qquad\qquad+O_\varepsilon\big(hX^{1+\varepsilon}(h/X)^{c/3}\big)+O_\varepsilon\Big(hX^{1+\varepsilon}\big(hX^{-(1-B_1)}\big)^{1/3(1-B_1)}\Big)
\end{eqnarray*}
uniformly for $X^{1-B_3}\ll h\ll X^{1-B_2}$.
\end{theoremC}

\begin{theoremD}\label{thirdtheorem}
Assume GRH. If $0<B_1\leq B_2< B_3<1/d_F$ and
\begin{eqnarray*}
\tilde{V}_F(X,h)=h X\Big(d_F \log\frac{X}{h}+\log\mathfrak{q}_F-(\gamma_0+\log 2\pi)d_F\Big)+O\big(hX^{1-c}\big)
\end{eqnarray*}
uniformly for  $X^{1-B_3}\ll h\ll X^{1-B_1}$ for some $c>0$, then we have
\begin{eqnarray*}
V_F(X, \delta)&=&\frac{1}{2}\delta X^2\Big(d_F \log\frac{1}{\delta}+\log\mathfrak{q}_F+(1-\gamma_0-\log2\pi)d_F\Big)\\
&&\qquad\qquad+O_\varepsilon\big(\delta^{1-\varepsilon}X^{2-c/3}\big)+O_\varepsilon\Big(\delta^{1-\varepsilon}X^{2}\big(\delta X^{B_3}\big)^{-2/3B_3}\Big)
\end{eqnarray*}
uniformly for $X^{-B_2}\ll\delta\ll X^{-B_1}$.
\end{theoremD}

\begin{theoremD}
Assume GRH. If $1/d_F<B_1\leq B_2< B_3<1$ and
\begin{eqnarray*}
\tilde{V}_F(X,h)=\frac{1}{6}h X\Big(6\log X-\big(3+8\log 2\big)\Big)+O\big(hX^{1-c}\big)
\end{eqnarray*}
uniformly for  $X^{1-B_3}\ll h\ll X^{1-B_1}$ for some $c>0$, then we have
\begin{eqnarray*}
V_F(X, \delta)=\frac{1}{6}\delta X^2\big(3\log X-4\log 2\big)+O_\varepsilon\big(\delta^{1-\varepsilon}X^{2-c/3}\big)+O_\varepsilon\Big(\delta^{1-\varepsilon}X^{2}\big(\delta X^{B_3}\big)^{-2/3B_3}\Big)
\end{eqnarray*}
uniformly for $X^{-B_2}\ll\delta\ll X^{-B_1}$.
\end{theoremD}

\begin{remark}
\emph{The main motivation for proving these theorems comes from the fact, shown in Sections \ref{Ratios} and \ref{Paircorr}, that the Selberg Orthogonality Conjecture and the ratios conjecture [\textbf{\ref{CFZ}}, \textbf{\ref{CS}}] for $F\in\mathcal{S}$ imply that 
\begin{eqnarray*}
\tilde{\mathcal{F}}_F(T^\alpha,T)&=&\left\{ \begin{array}{ll}
\frac{T\log X}{\pi}+O_\varepsilon(T^{\alpha/d_F+\varepsilon})+O_\varepsilon(T^{1/2+\varepsilon}) &\qquad \textrm{if $\alpha<d_F$,}\\
\frac{T}{\pi}\log \frac{ \mathfrak{q}_FT^{d_F}}{(2\pi)^{d_F}}-\frac{d_FT}{\pi}+O_\varepsilon(T^{1/2+\varepsilon}) & \qquad\textrm{if $\alpha>d_F$,} 
\end{array} \right.
\end{eqnarray*}
for a smoothed form of the pair correlation $\tilde{\mathcal{F}}_F(X,T)$ defined by
\begin{eqnarray*}
\tilde{\mathcal{F}}_F(X,T)=\sum_{-T\leq\gamma_F,\gamma'_F\leq T}X^{i(\gamma_F-\gamma'_F)}e^{-(\gamma_F-\gamma'_F)^2}.
\end{eqnarray*}
We expect that $\mathcal{F}_F(X,T)$ and $\tilde{\mathcal{F}}_F(X,T)$ satisfy the same estimates, at least up to some power saving error term, and these are the forms that appear in the theorems quoted above.  Alternatively, if we were to replace $\mathcal{F}_F(X,T)$ by $\tilde{\mathcal{F}}_F(X,T)$ in the statements of the above theorems, we would obtain correspondingly smoothed forms of the variances $V_F(X, \delta)$ and $\tilde{V}_F(X, h)$ instead; that is, variances involving averages with weight-functions whose mass is concentrated on $(1,X)$\footnote{For precise statements and proofs see [\textbf{\ref{D}}].}. We establish the form of the ratios conjecture we need in Section \ref{Ratios} and from this obtain the above formulae for $\tilde{\mathcal{F}}_F(X,T)$ in Section \ref{Paircorr}.}
\end{remark}

\begin{remark}
\emph{We draw attention in particular to the fact that when $d_F=1$ our theorems describe only one regime, but when $d_F\ge 2$ a new regime (described, for example, by Theorem A2) comes into play; the variances when $d_F\ge 2$ are therefore qualitatively different to when $d_F=1$.  We illustrate this in the following two figures, which show data from numerical computations.  In both cases we plot $\frac{\tilde{V}_F(X, h)}{h X}$ against $\log \frac{X}{h}$, for a fixed value of $X$ as $h$ varies and overlay the straight lines coming from the formulae for the variances described in the above theorems. In the first case, shown in Figure 1, $F$ is the Riemann zeta-function (so $\Lambda_F$ is just the von Mangoldt function)  and $X = 15 000 000$.  This is, of course, an example with $d_F=1$ and so one sees a single regime that is well described by (\ref{Zeta variance h LOT}).}

\begin{figure}\label{Figure 1} 
\includegraphics[scale=0.75]{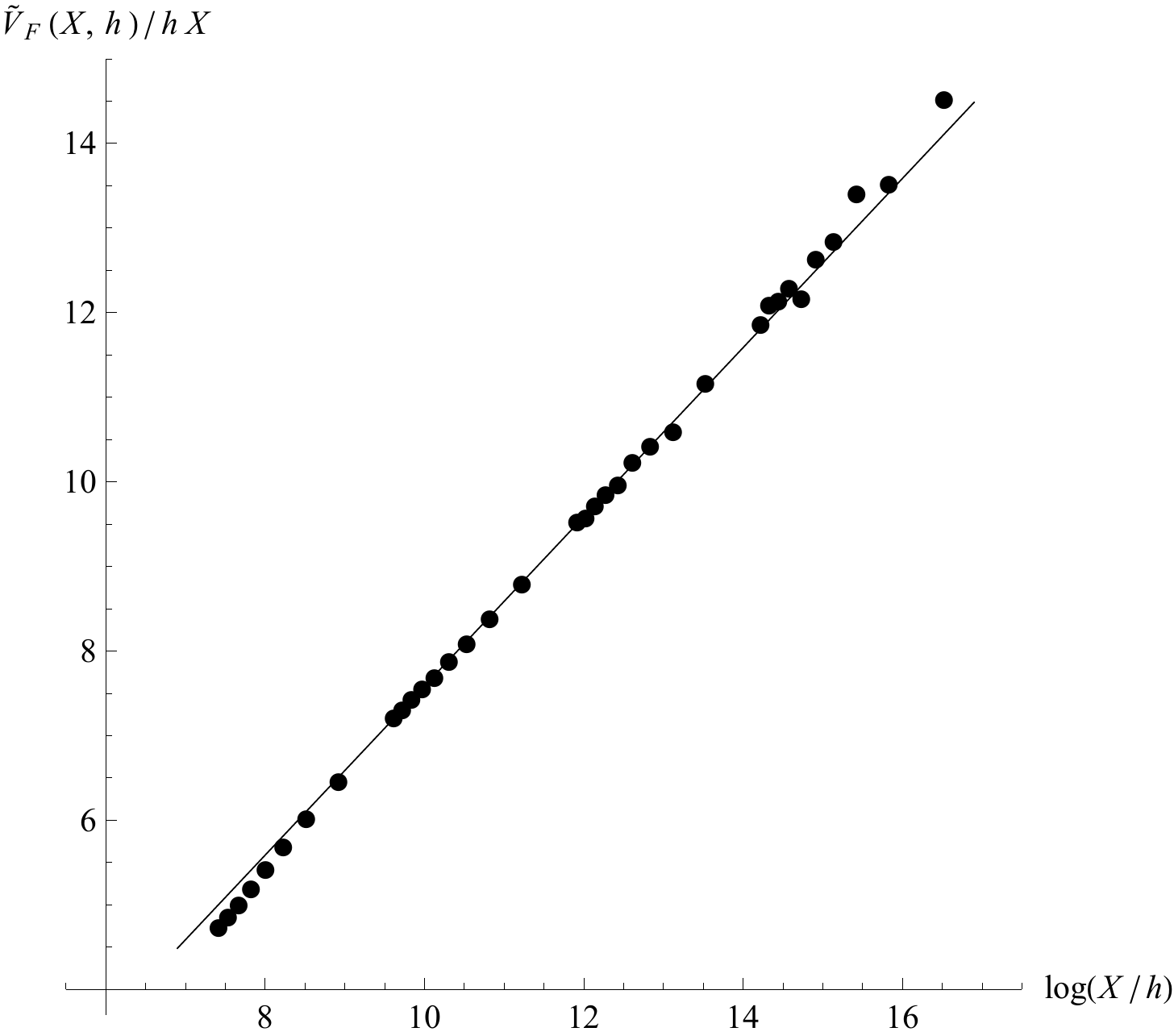}
\caption{$\frac{\tilde{V}_F(X, h)}{h X}$ plotted against $\log \frac{X}{h}$ when $F$ is the Riemann zeta-function and $X = 15 000 000$.  The line corresponds to (\ref{Zeta variance h LOT}). }
\end{figure}

\emph{By way of contrast, we plot in Figure 2 data for two $L$-functions with $d_F=2$.  In these examples $X = 1 000 000$.  The straight lines correspond to the formulae for the two regimes described by Theorems C1 and C2.}

\begin{figure}\label{Figure 2}
\includegraphics[scale=0.75]{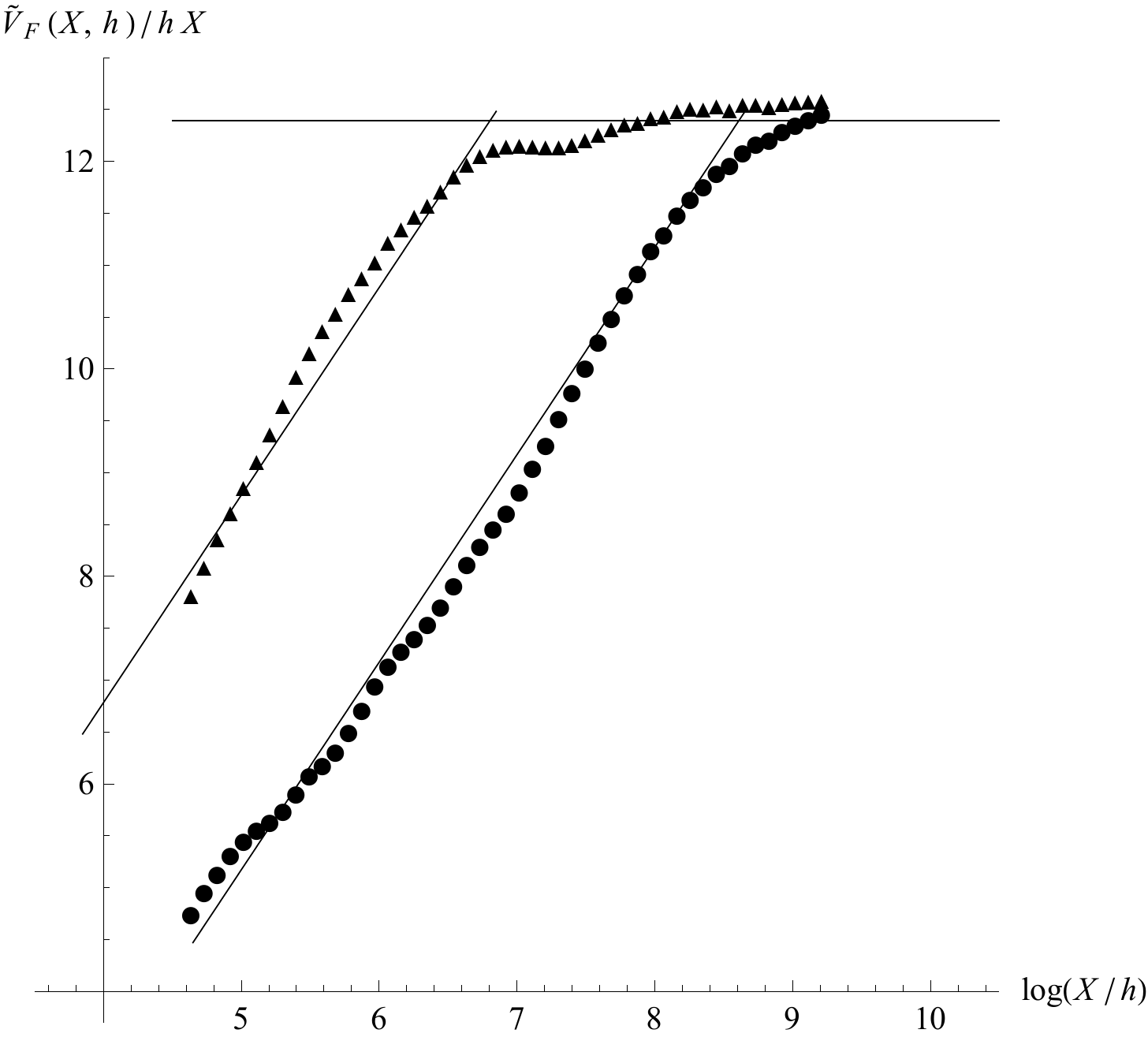}
\caption{$\frac{\tilde{V}_F(X, h)}{h X}$ plotted against $\log \frac{X}{h}$ when $F$ is associated with the Ramanujan $\tau$-function ($\bullet$) and with an elliptic curve of conductor 37 ($\blacktriangle$).  Here $X = 1 000 000$.  The lines correspond to the formulae for the two regimes described by Theorems C1 and C2. }
\end{figure}
\end{remark}

\begin{remark}
\emph{Note that, unlike the case of the Riemann zeta-function considered in [\textbf{\ref{GM}}], the A Theorems are not exactly the converse of the B Theorems, and the C Theorems are not exactly the converse of the D Theorems.  They are close to being the converse of each other, but with the power saving errors we have here, the intervals of uniformity do not match precisely.}
\end{remark}

The proofs of the theorems within each pair are essentially identical, so we only give the proofs of Theorems A1, B1 and C1.  Likewise, the proofs of Theorems D1 and D2 are similar to the proofs of C1 and C2, so we omit them too.

\section{Auxiliary lemmas}

\begin{lemma}\label{lot}
Suppose $f$ is a non-negative function with $f(t)\ll_\varepsilon |t|^\varepsilon$. If 
\[
\int_{-T}^{T}f(t)dt=T\big(\log T+A\big)+O\big(T^{1-c}\big)
\]
uniformly for $\kappa^{-(1-c_1)}\leq T\leq\kappa^{-(1+c_2)}$ for some  $A\in\mathbb{R}$ and $0<c,c_1,c_2<1$, then
\[
I(\kappa):=\int_{-\infty}^{\infty}\bigg(\frac{\sin\kappa u}{u}\bigg)^2f(u)du = \frac{\pi}{2}\kappa \Big(\log\frac{1}{\kappa}+B\Big)+O\big(\kappa^{1+c}\big)+O_\varepsilon\big(\kappa^{1+c_1-\varepsilon}\big)+O_{\varepsilon}\big(\kappa^{1+c_2-\varepsilon}\big)
\]
as $\kappa\rightarrow 0^{+}$, with $B=A+2-\gamma_0-\log 2$.
\end{lemma}
\begin{proof}
As in the proof of Lemma 2 of Goldston and Montgomery [\textbf{\ref{GM}}], we write 
\begin{eqnarray*}
I(\kappa)&=&\bigg(\int_{-U_1}^{U_1}\bigg)+\bigg(\int_{-U_2}^{-U_1}+\int_{U_1}^{U_2}\bigg)+\bigg(\int_{-\infty}^{-U_2}+\int_{U_2}^{\infty}\bigg)\\
&=&I_1(\kappa)+I_2(\kappa)+I_3(\kappa),
\end{eqnarray*}
say, where
\[
U_1=\kappa^{-(1-c_1)}\quad\textrm{and}\quad U_2=\kappa^{-(1+c_2)}.
\]

Since $f(t)\ll_\varepsilon |t|^\varepsilon$, we have
\begin{equation}\label{J1}
I_1(\kappa)\ll_\varepsilon \int_{-U_1}^{U_1}\kappa^2| u|^\varepsilon du\ll_\varepsilon \kappa^2U_{1}^{1+\varepsilon}\ll_{\varepsilon}\kappa^{1+c_1-\varepsilon}.
\end{equation}
Similarly,
\begin{equation}\label{J3}
I_3(\kappa)\ll_\varepsilon \int_{U_2}^{\infty}u^{-2+\varepsilon}du\ll_\varepsilon U_{2}^{-1+\varepsilon}\ll_{\varepsilon}\kappa^{1+c_2-\varepsilon}.
\end{equation}

To treat $I_2(\kappa)$ we let 
\[
r(t)=f(t)+f(-t)-\big(\log t+A+1\big)
\]
and
\[
R(u)=\int_{0}^{u}r(t)dt=\int_{0}^{u}\big(f(t)+f(-t)\big)dt-u\big(\log u+A\big).
\]
Then $R(u)\ll u^{1-c}$ uniformly for $U_1\leq u\leq U_2$, and
\begin{eqnarray*}
I_2(\kappa)&=&\int_{U_1}^{U_2}\bigg(\frac{\sin\kappa u}{u}\bigg)^2\big(f(u)+f(-u)\big)du\nonumber\\
&=&\int_{U_1}^{U_2}\bigg(\frac{\sin\kappa u}{u}\bigg)^2\big(\log u+A+1\big)du+\int_{U_1}^{U_2}\bigg(\frac{\sin\kappa u}{u}\bigg)^2dR(u).
\end{eqnarray*}
Integrating by parts, the second integral is
\begin{eqnarray*}
\ll \kappa^2 R(U_1)+U_{2}^{-2}R(U_2)+\int_{U_1}^{U_2}\big|R(u)\big|\bigg(\bigg|\frac{\kappa\sin2\kappa u}{u^2}\bigg|+\bigg|\frac{(\sin\kappa u)^2}{u^3}\bigg|\bigg)du\ll\kappa^{1+c}.
\end{eqnarray*}
For the first integral, we extend the range of integration to $[0,\infty)$. As in the treatment for $I_1(\kappa)$ and $I_3(\kappa)$, this introduces an error term of size $\ll_{\varepsilon}\kappa^{1+c_1-\varepsilon}+\kappa^{1+c_2-\varepsilon}$. Hence
\begin{equation}\label{J2}
I_2(\kappa)=\int_{0}^{\infty}\bigg(\frac{\sin\kappa u}{u}\bigg)^2\big(\log u+A+1\big)du+O\big(\kappa^{1+c}\big)+O_{\varepsilon}\big(\kappa^{1+c_1-\varepsilon}\big)+O_{\varepsilon}\big(\kappa^{1+c_2-\varepsilon}\big).
\end{equation}

In view of \eqref{J1}--\eqref{J2} we are left to estimate the main term, which is
\begin{eqnarray*}
&&\kappa\int_{0}^{\infty}\bigg(\frac{\sin u}{u}\bigg)^2\Big(\log u+\log\frac{1}{\kappa}+A+1\Big)du\\
&&\qquad\qquad=\frac{\pi}{2}\big(1-\gamma_0-\log 2\big)\kappa+\frac{\pi}{2}\kappa\Big(\log\frac{1}{\kappa}+A+1\Big)\\
&&\qquad\qquad=\frac{\pi}{2}\kappa\Big(\log\frac{1}{\kappa}+A+2-\gamma_0-\log 2\Big),
\end{eqnarray*}
and the lemma follows.
\end{proof}

\begin{lemma}\label{converselot}
Suppose $f,g$ are non-negative functions with $f(t)\ll_\varepsilon |t|^\varepsilon$. If 
\[
I(\kappa):=\int_{-\infty}^{\infty}\bigg(\frac{\sin\kappa u}{u}\bigg)^2f(u)du = \frac{\pi}{2}\kappa \Big(\log\frac{1}{\kappa}+B\Big)+O\big(\kappa^{1+c}g(T)\big)
\]
uniformly for $T^{-(1+c_1)}\leq \kappa\leq T^{-(1-c_2)}$ for some  $B\in\mathbb{R}$ and $0<c,c_1,c_2<1$, then
\[\int_{-T}^{T}f(t)dt = T\big(\log T+A\big)+O_\varepsilon\Big( \big(T^3g(T)\big)^{1/(3+c)+\varepsilon}\Big)+O_\varepsilon\big( T^{1-2c_1+\varepsilon}\big)+O_\varepsilon\big( T^{1-c_2/4+\varepsilon}\big)
\]
as $T\rightarrow \infty$, with $A=B-2+\gamma_0+\log 2$.
\end{lemma}
\begin{proof}
Let
\[
r(u)=f(u)+f(-u)-\big(\log u+B-1+\gamma_0+\log 2\big),
\]
and
\[
R(\kappa)=\int_{0}^{\infty}\bigg(\frac{\sin\kappa u}{u}\bigg)^2r(u)du.
\]
Then we have
\begin{eqnarray}\label{bound1}
R(\kappa)&=&I(\kappa)-\int_{0}^{\infty}\bigg(\frac{\sin\kappa u}{u}\bigg)^2\big(\log u+B-1+\gamma_0+\log 2\big)du\nonumber\\
&=&I(\kappa)- \frac{\pi}{2}\kappa \Big(\log\frac{1}{\kappa}+B\Big)\ll\kappa^{1+c}g(T)
\end{eqnarray}
uniformly for $T^{-(1+c_1)}\leq \kappa\leq T^{-(1-c_2)}$. Also, since $f(t)\ll_\varepsilon |t|^\varepsilon$, we get
\begin{equation}\label{bound2}
R(\kappa)\ll_\varepsilon\int_{0}^{\infty}\min\big\{\kappa^2,u^{-2}\big\}|u|^\varepsilon du\ll_\varepsilon \kappa^{1-\varepsilon}
\end{equation}
for all $\kappa\geq0$.

Let
\[
K_\eta(x)=\frac{\sin2\pi x+\sin2\pi(1+\eta)x}{2\pi x(1-4\eta^2x^2)}
\]
for $\eta>0$. Then
\begin{eqnarray*}
\hat{K_\eta}(t)= \left\{ \begin{array}{ll}
1 & \textrm{if $|t|\leq1,$}\\
\cos^2\Big(\frac{\pi(|t|-1)}{2\eta}\Big) & \textrm{if $1\leq|t|\leq1+\eta,$} \\
0 & \textrm{if $|t|\geq1+\eta$}.
\end{array} \right. 
\end{eqnarray*}
The kernel $K_\eta$ is even and satisfies the following properties: $K_\eta(x),\, K_\eta'(x)\rightarrow0$ as $x\rightarrow\infty$, and [\textbf{\ref{C}}]
\begin{eqnarray}\label{bound3}
K_\eta''(x)\ll\min\big\{1,\eta^{-3}|x|^{-3}\big\}.
\end{eqnarray}
Integrating by parts twice, we have
\[
\hat{K_\eta}(t)=\int_{0}^{\infty}K_\eta''(x)\bigg(\frac{\sin\pi tx}{\pi t}\bigg)^2dx.
\]
This implies that
\begin{eqnarray*}
\int_{0}^{\infty}r(t)\hat{K_\eta}\Big(\frac{t}{T}\Big)dt&=&\pi^{-2}T^2\int_{0}^{\infty}K_\eta''(x)R\Big(\frac{\pi x}{T}\Big)dx\\
&=&\pi^{-2}T^2\bigg(\int_{0}^{T_1}K_\eta''R+\int_{T_1}^{T_2}K_\eta''R+\int_{T_2}^{\infty}K_\eta''R\bigg),
\end{eqnarray*}
where $T_1=T^{-c_1}$ and $T_2=T^{c_2}$. From \eqref{bound2} and \eqref{bound3} we have
\[
\int_{0}^{T_1}K_\eta''R\ll_\varepsilon \int_{0}^{T_1}(x/T)^{1-\varepsilon}dx\ll_\varepsilon  T^{-(1+2c_1)+\varepsilon}
\]
and
\[
\int_{T_2}^{\infty}K_\eta''R\ll_\varepsilon \int_{T_2}^{\infty}\eta^{-3}x^{-3}(x/T)^{1-\varepsilon}dx\ll_\varepsilon \eta^{-3}T^{-(1+c_2)+\varepsilon}.
\]
Furthermore, \eqref{bound1} and \eqref{bound3} lead to
\[
\int_{T_1}^{T_2}K_\eta''R\ll \int_{T_1}^{T_2}\min\big\{1,\eta^{-3}x^{-3}\big\}(x/T)^{1+c}g(T)dx\ll \eta^{-(2+c)}T^{-(1+c)}g(T).
\]
So
\[
\int_{0}^{\infty}r(t)\hat{K_\eta}\Big(\frac{t}{T}\Big)dt\ll_\varepsilon T^{1-2c_1+\varepsilon}+\eta^{-3}T^{1-c_2+\varepsilon}+\eta^{-(2+c)}T^{1-c}g(T).
\]
Hence
\begin{eqnarray*}
\int_{-\infty}^{\infty}f(t)\hat{K_\eta}\Big(\frac{t}{T}\Big)dt&=&\int_{0}^{\infty}\big(\log t+B-1+\gamma_0+\log 2\big)\hat{K_\eta}\Big(\frac{t}{T}\Big)dt\\
&&\qquad\qquad+O_\varepsilon\big( T^{1-2c_1+\varepsilon}\big)+O_\varepsilon\big(\eta^{-3}T^{1-c_2+\varepsilon})+O\big(\eta^{-(2+c)}T^{1-c}g(T)\big)\\
&=&\int_{0}^{T}\big(\log t+B-1+\gamma_0+\log 2\big)dt+O\bigg(\int_{T}^{(1+\eta)T}\log t dt\bigg)\\
&&\qquad\qquad+O_\varepsilon\big( T^{1-2c_1+\varepsilon}\big)+O_\varepsilon\big(\eta^{-3}T^{1-c_2+\varepsilon})+O\big(\eta^{-(2+c)}T^{1-c}g(T)\big)\\
&=&T(\log T+B-2+\gamma_0+\log 2)+O_\varepsilon\big(\eta T^{1+\varepsilon}\big)\\
&&\qquad\qquad+O_\varepsilon\big( T^{1-2c_1+\varepsilon}\big)+O_\varepsilon\big(\eta^{-3}T^{1-c_2+\varepsilon})+O\big(\eta^{-(2+c)}T^{1-c}g(T)\big),
\end{eqnarray*}
and we obtain the lemma.
\end{proof}

\begin{lemma}\label{taub}
Suppose $f$ is a non-negative function. If
\[
\int_{-\infty}^{\infty}f(T+y)e^{-2|y|}dy=1+O\big(e^{-cY}\big)
\]
for $Y\leq T\leq Y+\log 2$ for some $c>0$, then
\[
\int_{0}^{\log 2}f(Y+y)e^{2y}dy=\frac{3}{2}+O\big(e^{-cY/2}\big).
\]
\end{lemma}
\begin{proof}
This is a special case of Lemma 1 of [\textbf{\ref{LPZ}}].
\end{proof}

\begin{lemma} \label{SVbound}
Assume GRH. We have
\begin{equation}\label{SV1}
\int_{1}^{X}\Big |\psi_F(x+\delta x)-\psi_F(x) - m_F\delta x\Big|^{2} dx\ll \delta X^2 \Big(\log \frac{2}{\delta} \Big)^{2}
\end{equation}
for $0<\delta\leq 1$, and
\begin{equation}\label{SV2}
\int_{1}^{X}\Big |\psi_F(x+h)-\psi_F(x) - m_Fh\Big|^{2} dx\ll hX \Big(\log \frac{2 X}{h} \Big)^{2} 
\end{equation}
for $0 < h \leq X$.
\end{lemma}
\begin{proof}
The argument is identical to that of Saffari and Vaughan in [\textbf{\ref{SV}}]. 
\end{proof}

\section{Ratios conjecture for $L$-functions in the Selberg class}\label{Ratios}

We would like to study
\[
R_F(\alpha,\beta,\gamma,\delta)=\int_{-T}^{T}\frac{F(s+\alpha)\overline{F}(1-s+\beta)}{F(s+\gamma)\overline{F}(1-s+\delta)}dt,
\]
where $s=1/2+it$, using the recipe in [\textbf{\ref{CFKRS}, \ref{CFZ}}]. The shifts are constrained as follows:
\begin{eqnarray}\label{constraints}
&& \big|\textrm{Re}(\alpha)\big|, \big|\textrm{Re}(\beta)\big| < \frac{1}{4},\nonumber\\
&& (\log T)^{-1}\ll \textrm{Re}(\gamma), \textrm{Re}(\delta)< \frac{1}{4}\\
&& \textrm{Im}(\alpha), \textrm{Im}(\beta), \textrm{Im}(\gamma), \textrm{Im}(\delta) \ll_\varepsilon T^{1-\varepsilon}.\nonumber
\end{eqnarray}

We use the approximate functional equation for the $L$-functions in the numerator,
\[
F(s)=\sum_{n}\frac{a_F(n)}{n^s}+X(s)\sum_{n}\frac{\overline{a_F}(n)}{n^{1-s}},
\]
and the normal Dirichlet series expansion for those in the denominator,
\[
F(s)^{-1}=\sum_{n}\frac{\mu_F(n)}{n^s}.
\]
As we integrate term-by-term, only the pieces with the same number of $X(s)$ as $\overline{X}(1-s)$ contribute to the main terms.

The terms from the first part of each approximate functional equation yield
\[
2T\sum_{hm=kn}\frac{a_F(m)\overline{a_F}(n)\mu_F(h)\overline{\mu_F}(k)}{m^{1/2+\alpha}n^{1/2+\beta}h^{1/2+\gamma}k^{1/2+\delta}}=2T\prod_p\bigg(\sum_{h+m=k+n}\frac{a_F(p^m)\overline{a_F}(p^n)\mu_F(p^h)\overline{\mu_F}(p^k)}{p^{(1/2+\alpha)m+(1/2+\beta)n+(1/2+\gamma)h+(1/2+\delta)k}}\bigg).
\]
We note that the functions $a_F(n)$, $\mu_F(n)$ are multiplicative because of the existence of the Euler product \eqref{Euler}, and
\[
b_F(p)=a_F(p)=-\mu_F(p).
\]
Hence the above expression is
\[
2TA_F(\alpha,\beta,\gamma,\delta)\frac{(F\otimes \overline{F})(1+\alpha+\beta)(F\otimes \overline{F})(1+\gamma+\delta)}{(F\otimes \overline{F})(1+\alpha+\delta)(F\otimes \overline{F})(1+\beta+\gamma)},
\]
where $A_F(\alpha,\beta,\gamma,\delta)$ is an arithmetical factor given by some Euler product that is absolutely and uniformly convergent in some product of fixed half-planes containing the origin,
\begin{eqnarray}\label{A}
A_F(\alpha,\beta,\gamma,\delta)&=&\prod_{p}\bigg(\sum_{h+m=k+n}\frac{a_F(p^m)\overline{a_F}(p^n)\mu_F(p^h)\overline{\mu_F}(p^k)}{p^{(1/2+\alpha)m+(1/2+\beta)n+(1/2+\gamma)h+(1/2+\delta)k}}\bigg)\\
&&\qquad\textrm{exp}\Bigg(\sum_{l=1}^{\infty}l\big|b_F(p^l)\big|^2\bigg(\frac{1}{p^{l(1+\alpha+\delta)}}+\frac{1}{p^{l(1+\beta+\gamma)}}-\frac{1}{p^{l(1+\alpha+\beta)}}-\frac{1}{p^{l(1+\gamma+\delta)}}\bigg)\Bigg).\nonumber
\end{eqnarray}
Here for any $F,G\in\mathcal{S}$, we define the tensor product $F\otimes G$ as in [\textbf{\ref{N}}]
\[
(F\otimes G)(s)=\prod_{p}\textrm{exp}\bigg(\sum_{l=1}^{\infty}\frac{lb_F(p^l)b_G(p^l)}{p^{ls}}\bigg).
\]

The contribution of the terms coming from the second part of each approximate functional equation is similar to the first piece except that $\alpha$ is replaced by $-\beta$, and $\beta$ is replaced by $-\alpha$. Also, because of the factor $X(s)$, we have an extra factor of
\[
X(s+\alpha)\overline{X}(1-s+\beta)=\Big(\frac{\mathfrak{q}_F(|t|+2)^{d_F}}{(2\pi)^{d_F}}\Big)^{-(\alpha+\beta)}\bigg(1+O\Big(\frac{1}{|t|+2}\Big)\bigg).
\]
Thus the recipe leads to the following ratios conjecture:

\begin{conjecture}\label{RCF}
With $\alpha,\beta,\gamma$ and $\delta$ satisfying \eqref{constraints} we have
\begin{eqnarray*}
&&\!\!\!\!\!\!\!\! R_F(\alpha,\beta,\gamma,\delta)=\int_{-T}^{T}\bigg(A_F(\alpha,\beta,\gamma,\delta)\frac{(F\otimes \overline{F})(1+\alpha+\beta)(F\otimes \overline{F})(1+\gamma+\delta)}{(F\otimes \overline{F})(1+\alpha+\delta)(F\otimes \overline{F})(1+\beta+\gamma)}\\
&&\!\!\!\!\!\!\!\!\qquad\qquad+\Big(\frac{\mathfrak{q}_F(|t|+2)^{d_F}}{(2\pi)^{d_F}}\Big)^{-(\alpha+\beta)}A_F(-\beta,-\alpha,\gamma,\delta)\frac{(F\otimes \overline{F})(1-\alpha-\beta)(F\otimes \overline{F})(1+\gamma+\delta)}{(F\otimes \overline{F})(1-\alpha+\gamma)(F\otimes \overline{F})(1-\beta+\delta)}\bigg)dt\\
&&\qquad\qquad\qquad\qquad+O_\varepsilon(T^{1/2+\varepsilon}),
\end{eqnarray*}
where $A_F(\alpha,\beta,\gamma,\delta)$ is defined as in \eqref{A}.
\end{conjecture}

We next investigate the analytic properties of $(F\otimes \overline{F})(s)$ at $s=1$. We have
\begin{eqnarray}\label{FxF}
\frac{(F\otimes \overline{F})'}{(F\otimes \overline{F})}(s)&=&-\sum_p\sum_{l=1}^{\infty}\frac{l^2|b_F(p^l)|^2(\log p)}{p^{ls}}=-\sum_p\frac{|b_F(p)|^2(\log p)}{p^{s}}+O(1)\nonumber\\
&=&-\sum_p\frac{|a_F(p)|^2(\log p)}{p^{s}}+O(1),
\end{eqnarray}
provided that $\textrm{Re}(s)>\frac{1}{2}$. Let 
\[
S(x)=\sum_{p\leq x}\frac{|a_F(p)|^2}{p}.
\]
The Selberg Orthogonality Conjecture says that
\[
S(x)=\log\log x+O(1).
\]
So for $\sigma_0>0$ and $|\sigma-\sigma_0|\leq\sigma_0/2$ ($\sigma\in\mathbb{C}$), partial summation gives
\[
\sum_{p\leq x}\frac{|a_F(p)|^2}{p^{1+\sigma}}=O\bigg(\frac{\log\log x}{x^{\textrm{Re}(\sigma)}}\bigg)+\sigma\int_{1}^{x}\frac{S(t)}{t^{\sigma+1}}dt=O\bigg(\frac{\log\log x}{x^{\textrm{Re}(\sigma)}}\bigg)+O(1)+\sigma\int_{1}^{x}\frac{\log\log t}{t^{\sigma+1}}dt.
\]
Taking $x\rightarrow\infty$ we obtain
\[
\sum_{p}\frac{|a_F(p)|^2}{p^{1+\sigma}}=O(1)+\sigma\int_{1}^{\infty}\frac{\log\log t}{t^{\sigma+1}}dt=O(1)-(\gamma_0+\log \sigma)=O(1)-\log \sigma.
\]
Hence using Cauchy's theorem we get
\[
\sum_{p}\frac{|a_F(p)|^2(\log p)}{p^{1+\sigma_0}}=\frac{1}{\sigma_0}+O(1).
\]
It follows from \eqref{FxF} that $(F\otimes \overline{F})(s)$ has a simple pole at $s=1$.

Note that for a function $f(u,v)$ analytic at $(u,v)=(\alpha,\alpha)$, a simple calculation shows that
\[
\frac{d}{d\alpha}\frac{f(\alpha,\gamma)}{(F\otimes \overline{F})(1-\alpha+\gamma)}\bigg|_{\gamma=\alpha}=\frac{f(\alpha,\alpha)}{r_{F\otimes \overline{F}}},
\]
where $r_{F\otimes \overline{F}}$ is the residue of $(F\otimes \overline{F})$ at $s=1$. It is also easy to verify that $A_F(\alpha,\beta,\alpha,\beta)=1$. So taking the derivatives of the expressions in Conjecture \ref{RCF} with respect to $\alpha$, $\beta$ and setting $\gamma=\alpha$, $\delta=\beta$ we have

\begin{conjecture}\label{RCd}
With $\alpha$ and $\beta$ satisfying \eqref{constraints} we have
\begin{eqnarray*}
&&\!\!\!\!\!\!\!\!\int_{-T}^{T}\frac{F'}{F}(s+\alpha)\frac{\overline{F}'}{\overline{F}}(1-s+\beta)dt=\int_{-T}^{T}\Bigg(\bigg(\frac{(F\otimes \overline{F})'}{(F\otimes \overline{F})}\bigg)'(1+\alpha+\beta)\\
&&\!\!\!\!\!\!\!\!\qquad\qquad+\frac{1}{r_{F\otimes \overline{F}}^2}\Big(\frac{\mathfrak{q}_F(|t|+2)^{d_F}}{(2\pi)^{d_F}}\Big)^{-(\alpha+\beta)}A_F(-\beta,-\alpha,\alpha,\beta)(F\otimes \overline{F})(1-\alpha-\beta)(F\otimes \overline{F})(1+\alpha+\beta)\\
&&\!\!\!\!\!\!\!\!\qquad\qquad\qquad\qquad+\frac{\partial^2}{d\alpha d\beta}A_F(\alpha,\beta,\gamma,\delta)\bigg|_{\gamma=\alpha,\delta=\beta}\ \Bigg)dt+O_\varepsilon(T^{1/2+\varepsilon}),
\end{eqnarray*}
where $A_F(\alpha,\beta,\gamma,\delta)$ is defined as in \eqref{A}.
\end{conjecture}

\section{Pair correlation of zeros of $L$-functions in the Selberg class}\label{Paircorr}

\subsection{The pair correlation function}

Let $F\in\mathcal{S}$. We want to evaluate the sum
\[
S(F)=\sum_{-T\leq \gamma_F,\gamma'_F\leq T}h(\gamma_F-\gamma'_F).
\] 

We follow the approach in [\textbf{\ref{CS}}] and compute this using contour integrals. Let $1/2<a<1$ and $\mathcal{C}$ be the positively oriented rectangle with vertices at $1-a-iT$, $a-iT$, $a+iT$ and $1-a+iT$. Then
\[
S(F)=\frac{1}{(2\pi i)^2}\int_{\mathcal{C}}\int_{\mathcal{C}}\frac{F'}{F}(u)\frac{F'}{F}(v)h\big(-i(u-v)\big)dudv.
\]
The horizontal contributions are small and can be ignored. We denote 
\[
S(F)=I_1+I_2+2I_3+O_\varepsilon(T^\varepsilon),
\] 
where $I_1$ has vertical parts $a$ and $a$, $I_2$ has vertical parts $1-a$ and $1-a$, and $I_3$ has vertical parts $a$ and $1-a$.

Using GRH and moving the contours to the right of $1$ we have $I_1=O_\varepsilon(T^\varepsilon)$.

For $I_2$ we use the functional equation 
\begin{equation}\label{fe}
\frac{F'}{F}(s)=\frac{X'}{X}(s)-\frac{\overline{F}'}{\overline{F}}(1-s).
\end{equation}
Here
\begin{eqnarray*}
\frac{X'}{X}(s)&=&-2\log Q-\sum_{j=1}^{r}\lambda_j\bigg(\frac{\Gamma'}{\Gamma}\big(\lambda_js+\mu_j\big)+\frac{\Gamma'}{\Gamma}\big(\lambda_j(1-s)+\overline{\mu_j}\big)\bigg)\\
&=&-\log\frac{ \mathfrak{q}_F(|t|+2)^{d_F}}{(2\pi)^{d_F}}+O\bigg(\frac{1}{|t|+2}\bigg).
\end{eqnarray*}
We apply \eqref{fe} to both $F'/F(u)$ and $F'/F(v)$. For the terms involving $\overline{F}'/\overline{F}(1-u)$ or $\overline{F}'/\overline{F}(1-v)$, we move the corresponding contour to the right of $1$, and as in the treatment for $I_1$, we get $O_\varepsilon(T^\varepsilon)$. For the term with $X'/X(u)$ and $X'/X(v)$, we move both contours to $\textrm{Re}(u)=\textrm{Re}(v)=\frac{1}{2}$. Again that introduces an error term of size $O_\varepsilon(T^\varepsilon)$. Hence
\begin{eqnarray*}
I_2&=&\frac{1}{(2\pi)^2}\int_{-T}^{T}\int_{-T}^{T}\frac{X'}{X}(\tfrac{1}{2}+iu)\frac{X'}{X}(\tfrac{1}{2}+iv)h(u-v)dudv+O_\varepsilon(T^\varepsilon)\\
&=&\frac{1}{(2\pi)^2}\int_{-T}^{T}\int_{-T}^{T}\log \frac{ \mathfrak{q}_F(|u|+2)^{d_F}}{(2\pi)^{d_F}}\log \frac{ \mathfrak{q}_F(|v|+2)^{d_F}}{(2\pi)^{d_F}}h(u-v)dudv+O_\varepsilon(T^\varepsilon)\\
&=&\frac{2}{(2\pi)^2}\int_{-T}^{T}\int_{v}^{T}\log \frac{ \mathfrak{q}_F(|u|+2)^{d_F}}{(2\pi)^{d_F}}\log\frac{ \mathfrak{q}_F(|v|+2)^{d_F}}{(2\pi)^{d_F}}h(u-v)dudv+O_\varepsilon(T^\varepsilon),
\end{eqnarray*}
as $h$ is even. Changing the variables $t=v$ and $\eta=u-v$ we get
\begin{eqnarray*}
I_2=\frac{2}{(2\pi)^2}\int_{0}^{2T}h(\eta)\int_{-T}^{T-\eta}\log \frac{ \mathfrak{q}_F(|t+\eta|+2)^{d_F}}{(2\pi)^{d_F}}\log \frac{ \mathfrak{q}_F(|t|+2)^{d_F}}{(2\pi)^{d_F}}dtd\eta+O_\varepsilon(T^\varepsilon).
\end{eqnarray*}
We can extend the inner integral to $t=T$ introducing an error term of size $\ll (\log T)^2\int_{0}^{2T}\eta h(\eta)d\eta\ll (\log T)^3$. The same argument shows that the term $\log \frac{ \mathfrak{q}_F(|t+\eta|+2)^{d_F}}{(2\pi)^{d_F}}$ can be replaced by $\log \frac{ \mathfrak{q}_F(|t|+2)^{d_F}}{(2\pi)^{d_F}}$ with the same error term. So
\begin{eqnarray*}
I_2&=&\frac{2}{(2\pi)^2}\int_{0}^{2T}h(\eta)\int_{-T}^{T}\bigg(\log \frac{ \mathfrak{q}_F(|t|+2)^{d_F}}{(2\pi)^{d_F}}\bigg)^2dtd\eta+O_\varepsilon(T^\varepsilon)\\
&=&\frac{1}{(2\pi)^2}\int_{-T}^{T}\int_{-2T}^{2T}h(\eta)\bigg(\log \frac{ \mathfrak{q}_F(|t|+2)^{d_F}}{(2\pi)^{d_F}}\bigg)^2d\eta dt+O_\varepsilon(T^\varepsilon).
\end{eqnarray*}

We next consider 
\[
I_3=-\frac{1}{(2\pi i)^2}\int_{a-iT}^{a+iT}\int_{1-a-iT}^{1-a+iT}\frac{F'}{F}(u)\frac{F'}{F}(v)h\big(-i(u-v)\big)dudv.
\]
Letting $u-v=i\eta$ we get
\begin{eqnarray*}
I_3=-\frac{1}{(2\pi )^2i}\int_{-2T-i(1-2a)}^{2T-i(1-2a)}h(\eta)\int_{a-iT_1}^{a+iT_2}\frac{F'}{F}(v)\frac{F'}{F}(v+i\eta)dvd\eta,
\end{eqnarray*} 
where 
\[
T_1=\textrm{min}\big\{T,T+\textrm{Re}(\eta)\big\}\quad\textrm{and}\quad T_2=\textrm{min}\big\{T,T-\textrm{Re}(\eta)\big\}.
\]
We now use the functional equation \eqref{fe} for $F'/F(v+i\eta)$. The term with $X'/X(v+i\eta)$ is $O_\varepsilon(T^\varepsilon)$ by moving the $v$-contour to the right of $1$. Thus,
\begin{eqnarray*}
I_3&=&\frac{1}{(2\pi )^2i}\int_{-2T-i(1-2a)}^{2T-i(1-2a)}h(\eta)\int_{a-iT_1}^{a+iT_2}\frac{F'}{F}(v)\frac{\overline{F}'}{\overline{F}}(1-v-i\eta)dvd\eta+O_\varepsilon(T^\varepsilon)\\
&=&\frac{1}{(2\pi )^2}\int_{-2T-i(1-2a)}^{2T-i(1-2a)}h(\eta)\int_{-T_1}^{T_2}\frac{F'}{F}\big(s+(a-\tfrac{1}{2})\big)\frac{\overline{F}'}{\overline{F}}\big(1-s+(\tfrac{1}{2}-a-i\eta)\big)dtd\eta+O_\varepsilon(T^\varepsilon),
\end{eqnarray*} 
where $s=1/2+it$.

In view of Conjecture \ref{RCd}, we have
\begin{eqnarray}\label{100}
I_3=\frac{1}{(2\pi )^2}\int_{-2T-i(1-2a)}^{2T-i(1-2a)}h(\eta)\int_{-T_1}^{T_2}g(-\eta,t)dtd\eta+O_\varepsilon(T^{1/2+\varepsilon}),
\end{eqnarray}
where
\begin{eqnarray*}
&&g(\eta,t)=\bigg(\frac{(F\otimes \overline{F})'}{(F\otimes \overline{F})}\bigg)'(1+i\eta)+\frac{1}{r_{F\otimes \overline{F}}^2}\Big(\frac{ \mathfrak{q}_F(|t|+2)^{d_F}}{(2\pi)^{d_F}}\Big)^{-i\eta}\\
&&\quad\qquad A_F\big(-\tfrac{1}{2}+a-i\eta,-a+\tfrac{1}{2},a-\tfrac{1}{2},\tfrac{1}{2}-a+i\eta\big)\\
&&\qquad\qquad\quad\quad(F\otimes \overline{F})(1-i\eta)(F\otimes \overline{F})(1+i\eta)+\frac{\partial^2}{d\alpha d\beta}A_F(\alpha,\beta,\gamma,\delta)\bigg|_{\gamma=\alpha=a-\tfrac{1}{2},\delta=\beta=\tfrac{1}{2}-a+i\eta}.
\end{eqnarray*}
A simple calculation shows that 
\[
A_F\big(-\tfrac{1}{2}+a-i\eta,-a+\tfrac{1}{2},a-\tfrac{1}{2},\tfrac{1}{2}-a+i\eta\big)=A_F(i\eta),
\]
where
\begin{eqnarray}\label{A_F}
A_F(r)&=&\prod_{p}\bigg(\sum_{h+m=k+n}\frac{a_F(p^m)\overline{a_F}(p^n)\mu_F(p^h)\overline{\mu_F}(p^k)}{p^{-r m+n+(1+r)k}}\bigg)\nonumber\\
&&\qquad\qquad\qquad\textrm{exp}\Bigg(\sum_{l=1}^{\infty}l\big|b_F(p^l)\big|^2\bigg(\frac{2}{p^{l}}-\frac{1}{p^{l(1-r)}}-\frac{1}{p^{l(1+r)}}\bigg)\Bigg),
\end{eqnarray}
and
\[
\frac{\partial^2}{d\alpha d\beta}A_F(\alpha,\beta,\gamma,\delta)\bigg|_{\gamma=\alpha=a-\tfrac{1}{2},\delta=\beta=\tfrac{1}{2}-a+i\eta}=-B_F(i\eta),
\]
where
\begin{eqnarray}\label{B_F}
\!\!\!\!\!\!B_F(r)&=&\sum_{p}(\log p)^2\bigg(-\sum_{h+m=k+n}\frac{a_F(p^m)\overline{a_F}(p^n)\mu_F(p^h)\overline{\mu_F}(p^k)mn}{p^{(n+k)(1+r)}}+\sum_{l=1}^{\infty}\frac{l^3\big|b_F(p^l)\big|^2}{p^{l(1+r)}}\bigg).
\end{eqnarray}
So
\begin{eqnarray*}
g(\eta,t)&=&\bigg(\frac{(F\otimes \overline{F})'}{(F\otimes \overline{F})}\bigg)'(1+i\eta)+\frac{1}{r_{F\otimes \overline{F}}^2}\Big(\frac{ \mathfrak{q}_F(|t|+2)^{d_F}}{(2\pi)^{d_F}}\Big)^{-i\eta}A_F(i\eta)\\
&&\qquad\qquad(F\otimes \overline{F})(1-i\eta)(F\otimes \overline{F})(1+i\eta)-B_F(i\eta).
\end{eqnarray*}
As before, we can extend the range of the inner integral in \eqref{100} to $[-T,T]$ producing an error term of size $O_\varepsilon( T^\varepsilon)$. Hence
\[
I_3=\frac{1}{(2\pi )^2}\int_{-T}^{T}\int_{-2T-i(1-2a)}^{2T-i(1-2a)}h(\eta)g(-\eta,t)d\eta dt+O_\varepsilon(T^{1/2+\varepsilon}).
\]

Next we move the path of integration of the inner integral to the real axis from $-2T$ to $2T$ with a principal value as we pass though $0$. Note that $A_F'(0)=0$, so near $\eta=0$ we have
\begin{eqnarray*}
g(\eta,t)=-\frac{i}{\eta}\log \frac{ \mathfrak{q}_F(|t|+2)^{d_F}}{(2\pi)^{d_F}}+O(1).
\end{eqnarray*}
Thus
\[
I_3=\frac{h(0)}{4\pi}\int_{-T}^{T}\log\frac{ \mathfrak{q}_F(|t|+2)^{d_F}}{(2\pi)^{d_F}}dt+\frac{1}{(2\pi )^2}\int_{-T}^{T}\int_{-2T}^{2T}h(\eta)g(\eta,t)d\eta dt+O_\varepsilon(T^{1/2+\varepsilon}),
\]
after changing the variable $\eta$ to $-\eta$. Summing up we have

\begin{conjecture}\label{RCS}
For $h$ a suitable even test function we have
\begin{eqnarray*}
&&\!\!\!\!\!\!\!\!\!\!\!\!\sum_{-T\leq\gamma_F,\gamma'_F\leq T}h(\gamma_F-\gamma'_F)=\frac{h(0)}{2\pi}\int_{-T}^{T}\log\frac{ \mathfrak{q}_F(|t|+2)^{d_F}}{(2\pi)^{d_F}}dt+\frac{1}{(2\pi)^2}\int_{-T}^{T}\int_{-2T}^{2T}h(\eta)\\
&&\qquad\bigg[\bigg(\log \frac{ \mathfrak{q}_F(|t|+2)^{d_F}}{(2\pi)^{d_F}}\bigg)^2+2\bigg(\bigg(\frac{(F\otimes \overline{F})'}{(F\otimes \overline{F})}\bigg)'(1+i\eta)+\frac{1}{r_{F\otimes \overline{F}}^2}\Big(\frac{ \mathfrak{q}_F(|t|+2)^{d_F}}{(2\pi)^{d_F}}\Big)^{-i\eta}\\
&&\qquad\qquad\qquad A_F(i\eta)(F\otimes \overline{F})(1-i\eta)(F\otimes \overline{F})(1+i\eta)-B_F(i\eta)\bigg)\bigg]d\eta dt+O_\varepsilon(T^{1/2+\varepsilon}),
\end{eqnarray*}
where $A_F(r)$ and $B_F(r)$ are defined as in \eqref{A_F} and \eqref{B_F}.
\end{conjecture}

\subsection{The form factor}

Throughout this section, we shall denote
\[
X=T^\alpha,\qquad\ell = \log \frac{ \mathfrak{q}_F(|t|+2)^{d_F}}{(2\pi)^{d_F}}\qquad\textrm{and}\qquad \mathcal{L} = \log \frac{ \mathfrak{q}_FT^{d_F}}{(2\pi)^{d_F}}.
\]

We recall that
\begin{eqnarray*}
\tilde{\mathcal{F}}_F(X,T)&=&\sum_{-T\leq\gamma_F,\gamma'_F\leq T}X^{i(\gamma_F-\gamma'_F)}e^{-(\gamma_F-\gamma'_F)^2}\\
&=&\sum_{-T\leq\gamma_F,\gamma'_F\leq T}\cos\big((\gamma_F-\gamma'_F)\log X\big)e^{-(\gamma_F-\gamma'_F)^2}.
\end{eqnarray*}
The function $\tilde{\mathcal{F}}_F(X,T)$ is in a suitable form to apply Conjecture \ref{RCS} with
\[
h(\eta)=\cos(\eta\log X)e^{-\eta^2},
\]
and using that we shall write
\[
\tilde{\mathcal{F}}_F(X,T)=\sum_{-T\leq\gamma_F,\gamma'_F\leq T}h(\gamma_F-\gamma'_F)=J_1+J_2+O_\varepsilon(T^{1/2+\varepsilon}).
\]

Since $h$ is even, we have
\[
\int_{-2T}^{2T}\eta^{2k-1}h(\eta)d\eta  =0,
\]
and
\begin{eqnarray*}
&&\int_{-2T}^{2T}\eta^{2k}h(\eta)d\eta =\sum_{j=0}^{\infty}\frac{(-1)^j(\log X)^{2j}}{(2j)!}\int_{-2T}^{2T}\eta^{2(k+j)}e^{-\eta^2}d\eta\\
&&\qquad\qquad=\sqrt{\pi}\sum_{j=0}^{\infty}\frac{(-1)^j(2k+2j)!}{2^{2k+2j}(2j)!(k+j)!}(\log X)^{2j}+O\Big((2T)^{2k-1}\exp\big(2(\log X)T-4T^2\big)\Big)
\end{eqnarray*}
for any $k\in\mathbb{Z}$. In particular,
\[
\int_{-2T}^{2T}\eta^{2k}h(\eta)d\eta \ll \big(\log X/2\big)^{2k}\exp\big(-(\log X)^2/4\big)+(2T)^{2k-1}\exp\big(2(\log X)T-4T^2\big)
\]
for any $k\geq 0$.

Moreover
\[
\int_{-2T}^{2T}\eta^{2k-1}\cos(\eta\ell)h(\eta)d\eta =0,
\]
and
\begin{eqnarray*}
\int_{-2T}^{2T}\eta^{2k}\cos(\eta\ell)h(\eta)d\eta &=&\sum_{i,j=0}^{\infty}\frac{(-1)^{i+j}\ell^{2i}(\log X)^{2j}}{(2i)!(2j)!}\bigg(\int_{-2T}^{2T}\eta^{2(k+i+j)}e^{-\eta^2}d\eta\bigg)\\
&=&\sqrt{\pi}\sum_{i,j=0}^{\infty}\frac{(-1)^{i+j}(2k+2i+2j)!}{2^{2k+2i+2j}(2i)!(2j)!(k+i+j)!}\ell^{2i}(\log X)^{2j}\\
&&\qquad\qquad+O\Big((2T)^{2k-1}\exp\big(2(\log X+\mathcal{L})T-4T^2\big)\Big)
\end{eqnarray*}
for any $k\in\mathbb{Z}$. In particular,
\begin{eqnarray*}
\int_{-2T}^{2T}\eta^{2k}\cos(\eta\ell)h(\eta)d\eta &\ll& \big((\log X+\ell)/2\big)^{2k}\exp\big(-(\log X-\ell)^2/4\big)\\
&&\qquad\qquad+(2T)^{2k-1}\exp\big(2(\log X+\mathcal{L})T-4T^2\big)
\end{eqnarray*}
for any $k\geq 0$, and hence
\begin{eqnarray*}
&&\!\!\!\!\!\!\!\!\int_{-T}^{T}\int_{-2T}^{2T}\eta^{2k}\cos(\eta\ell)h(\eta)d\eta dt\\
&&\!\!\!\!\!\!\!\!\qquad\ll_\varepsilon \left\{ \begin{array}{ll}
T^{\alpha/d_F+\varepsilon}\mathcal{L}^{2k}+(2T)^{2k}\exp\big(2(\log X+\mathcal{L})T-4T^2\big) & \textrm{if $\alpha<d_F,$}\\
T(\log X)^{2k}\exp\big(-c(\log X)^2\big)+(2T)^{2k}\exp\big(2(\log X+\mathcal{L})T-4T^2\big) & \textrm{if $\alpha>d_F$} 
\end{array} \right. 
\end{eqnarray*}
with some absolute constant $c>0$, for any $k\geq 0$.

Similarly,
\[
\int_{-2T}^{2T}\eta^{2k}\sin(\eta\ell)h(\eta)d\eta  =0,
\]
and
\begin{eqnarray*}
&&\!\!\!\!\!\!\!\!\int_{-T}^{T}\int_{-2T}^{2T}\eta^{2k+1}\sin(\eta\ell)h(\eta)d\eta dt\\
&&\!\!\!\!\!\!\!\!\qquad\ll_\varepsilon \left\{ \begin{array}{ll}
T^{\alpha/d_F+\varepsilon}\mathcal{L}^{2k+1}+(2T)^{2k+1}\exp\big(2(\log X+\mathcal{L})T-4T^2\big) & \textrm{if $\alpha<d_F,$}\\
T(\log X)^{2k+1}\exp\big(-c(\log X)^2\big)+(2T)^{2k+1}\exp\big(2(\log X+\mathcal{L})T-4T^2\big) & \textrm{if $\alpha>d_F$} 
\end{array} \right. 
\end{eqnarray*}
with some absolute constant $c>0$, for any $k\geq 0$.

Expanding various terms in Conjecture \ref{RCS} we have
\begin{eqnarray*}
&&\bigg(\frac{(F\otimes \overline{F})'}{(F\otimes \overline{F})}\bigg)'(1+i\eta)=-\frac{1}{\eta^2}+O(1),\\
&&(F\otimes \overline{F})(1-i\eta)(F\otimes \overline{F})(1+i\eta)=\frac{r_{F\otimes \overline{F}}^2}{\eta^2}+O(1),\\
&&A_F(i\eta)=1+O(\eta^2),\\
&&B_F(ir)=O(1).
\end{eqnarray*}
So
\begin{eqnarray*}
J_2&=&\frac{1}{(2\pi)^2}\int_{-T}^{T}\int_{-2T}^{2T}\big(\ell^2-2\eta^{-2}+2\eta^{-2}\cos(\eta\ell)\big)h(\eta)d\eta dt+E\\
&=&\frac{2}{\pi\sqrt{\pi}}\int_{-T}^{T}\sum_{i=2}^{\infty}\sum_{j=0}^{\infty}\frac{(-1)^{i+j}(2i+2j-2)!}{2^{2i+2j}(2i)!(2j)!(i+j-1)!}\ell^{2i}(\log X)^{2j}dt+E,
\end{eqnarray*}
where
\[
E\ll_{\varepsilon,A} \left\{ \begin{array}{ll}
T^{\alpha/d_F+\varepsilon} & \textrm{if $\alpha<d_F,$}\\
T^{-A} & \textrm{if $\alpha>d_F$} 
\end{array} \right.
\]
for every $A>0$. The double sum in the integral equals
\begin{eqnarray*}
&&\!\!\!\!\!\!\!\!\!\!-\frac{\sqrt{\pi}}{8}\bigg(\big|\log X-\ell\big|\textrm{Erf}\Big(\frac{|\log X-\ell|}{2}\Big)+\big(\log X+\ell\big)\textrm{Erf}\Big(\frac{\log X+\ell}{2}\Big)\bigg)\\
&&\!\!\!\!\!\!\!\!\!\!\qquad\qquad+\frac{\sqrt{\pi}}{4}(\log X)\textrm{Erf}\Big(\frac{\log X}{2}\Big) +O\Big(\exp\big(-(\log X-\ell)^2/4\big)\Big)+O\Big(\mathcal{L}^2\exp\big(-(\log X)^2/4\big)\Big)\\
&&\!\!\!\!\!\!\!\!\ =-\frac{\sqrt{\pi}}{4}\Big(\max\big\{\log X,\ell\big\}-\log X\Big)+O\Big(\mathcal{L}^2\exp\big(-(\log X-\ell)^2/4\big)+\mathcal{L}^2\exp\big(-(\log X)^2/4\big)\Big).
\end{eqnarray*}
Hence
\[
J_2=-\frac{1}{2\pi}\int_{-T}^{T}\Big(\max\big\{\log X,\ell\big\}-\log X\Big)dt+E.
\]

On the other hand, 
\[
J_1=\frac{1}{2\pi}\int_{-T}^{T}\ell dt.
\]
Thus
\begin{eqnarray*}
\tilde{\mathcal{F}}_F(X,T)&=&\left\{ \begin{array}{ll}
\frac{T\log X}{\pi}+O_\varepsilon(T^{\alpha/d_F+\varepsilon})+O_\varepsilon(T^{1/2+\varepsilon}) &\qquad \textrm{if $\alpha<d_F$,}\\
\frac{T\mathcal{L}}{\pi}-\frac{d_FT}{\pi}+O_\varepsilon(T^{1/2+\varepsilon}) & \qquad\textrm{if $\alpha>d_F$.} 
\end{array} \right.
\end{eqnarray*}

\begin{conjecture}
We have
\[
\tilde{\mathcal{F}}_F(X,T)=\frac{T\log X}{\pi}+O_\varepsilon(X^{1/d_F+\varepsilon})+O_\varepsilon(T^{1/2+\varepsilon})
\]
uniformly for $T^{A_1}\leq X\leq T^{A_2}$ for any fixed $0<A_1\leq A_2< d_F$, and
\begin{equation*}
\tilde{\mathcal{F}}_F(X,T)=\frac{T}{\pi}\bigg(d_F\log \frac{T}{2\pi}+\log\mathfrak{q}_F-d_F\bigg)+O_\varepsilon(T^{1/2+\varepsilon})
\end{equation*}
uniformly for $T^{A_1}\leq X\leq T^{A_2}$ for any fixed $d_F<A_1\leq A_2<\infty$.
\end{conjecture}

\section{Proofs of main theorems}

\subsection{Proof of Theorem A\ref{maintheo}}

We begin by considering
\begin{eqnarray*}
I(X,T)&=&\int_{- T}^{T} \bigg| \sum_{|\gamma_{F}|\leq Z} \frac{X^{i \gamma_{F}}}{1 + (t - \gamma_{F})^{2}} \bigg|^{2}  dt\\
&=&\sum_{-Z\leq\gamma_F,\gamma'_F\leq Z}X^{i(\gamma_F-\gamma'_F)}\int_{-T}^{T}\frac{dt}{\big(1+(t-\gamma_F)^2\big)\big(1+(t-\gamma'_F)^2\big)},
\end{eqnarray*}
with $X,Z\geq T$. Using the fact that $N_F(t+1)-N_F(t)\ll \log(|t|+2)$, we can restrict the summation over the zeros to $-T\leq \gamma_F,\gamma'_F\leq T$ with an error term of size $\ll (\log T)^2$. Similarly, the range of the integration can be extended to $(-\infty,\infty)$ introducing an error term of size $\ll (\log T)^3$. So
\begin{eqnarray*}
I(X,T)&=&\sum_{-T\leq \gamma_F,\gamma'_F\leq T}X^{i(\gamma_F-\gamma'_F)}\int_{-\infty}^{\infty}\frac{dt}{\big(1+(t-\gamma_F)^2\big)\big(1+(t-\gamma'_F)^2\big)}+O\big((\log T)^3\big)\\
&=&\frac{\pi}{2}\, \mathcal{F}_F(X,T)+O\big((\log T)^3\big),
\end{eqnarray*}
and hence from \eqref{250} we have
\begin{equation*}
I(X,T)=\frac{T}{2}\Big(d_F\log \frac{T}{2\pi}+\log\mathfrak{q}_F-d_F\Big)+O\big(T^{1-c}\big)
\end{equation*}
uniformly for $X^{1/A_2}\ll T\ll X^{1/A_1}$. 

Let
\[
a(s)=\frac{(1+\delta)^s-1}{s}.
\]
Then
\[
\big|a(it)\big|^2=4\bigg(\frac{\sin\kappa t}{t}\bigg)^2,
\]
where $\kappa=\frac{\log(1+\delta)}{2}$. So by Lemma \ref{lot} we deduce that
\begin{eqnarray*}
&&\!\!\!\!\!\!\!\!\!\!\int_{-\infty}^{\infty}\big|a(it)\big|^2\bigg| \sum_{|\gamma_{F}|\leq Z} \frac{X^{i \gamma_{F}}}{1 + (t - \gamma_{F})^{2}} \bigg|^{2}dt\\
&&=\pi \kappa\Big(d_F\log\frac{1}{\kappa}+\log\mathfrak{q}_F+(1-\gamma_0-\log 4\pi)d_F\Big)+O\big(\kappa^{1+c}\big)+O_{\varepsilon}\big(\kappa^{1+c_1-\varepsilon}\big)+O_{\varepsilon}\big(\kappa^{1+c_2-\varepsilon}\big)\\
&&=\frac{\pi}{2}\delta\Big(d_F\log\frac{1}{\delta}+\log\mathfrak{q}_F+(1-\gamma_0-\log2\pi)d_F\Big)+O\big(\delta^{1+c}\big)+O_{\varepsilon}\big(\delta^{1+c_1-\varepsilon}\big)+O_{\varepsilon}\big(\delta^{1+c_2-\varepsilon}\big).
\end{eqnarray*}
The values of $T$ for which we have used Lemma \ref{lot} lie in the range
\[
\delta^{-(1-c_1)}\ll T\ll\delta^{-(1+c_2)}
\]
for some $0<c_1,c_2<1$.

Let $J$ be the above integral and $K$ be the same integral with $a(it)$ being replaced by $a(\frac{1}{2}+i\gamma_F)$. We write $J=\int|A|^2$ and $K=\int|B|^2$. Direct calculation shows that
\[
a(s)\ll\min\big\{\delta,1/|s|\big\}\quad\textrm{and}\quad a'(s)\ll\min\big\{\delta^2,\delta/|s|\big\}
\]
for $|\sigma|\leq 1$. Hence, since $N_F(t+1)-N_F(t)\ll \log(|t|+2)$,
\[
A, B\ll \min\big\{\delta,1/|t|\big\}\log(|t|+2)
\]
and
\[
a(it)-a\big(\tfrac{1}{2}+i\gamma_F\big)\ll\big(1+|t-\gamma_F|\big)\, \min\big\{\delta^2,\delta/|t|\big\}.
\]
Thus 
\[
A-B\ll \min\big\{\delta^2,\delta/|t|\big\}\big(\log(|t|+2)\big)^2,
\]
and hence
\[
|A|^2-|B|^2\ll \min\big\{\delta^3,\delta/|t|^2\big\}\big(\log(|t|+2)\big)^3,
\]
so that
\[
J-K\ll \delta^2\bigg(\log\frac{1}{\delta}\bigg)^3.
\]
It follows that
\begin{eqnarray}\label{101}
\int_{-\infty}^{\infty}\bigg| \sum_{|\gamma_{F}|\leq Z} \frac{a(\rho_F)X^{i \gamma_{F}}}{1 + (t - \gamma_{F})^{2}} \bigg|^{2}dt&=&\frac{\pi}{2}\delta\Big(d_F\log\frac{1}{\delta}+\log\mathfrak{q}_F+(1-\gamma_0-\log2\pi)d_F\Big)\\
&&\qquad\qquad+O\big(\delta^{1+c}\big)+O_\varepsilon\big(\delta^{1+c_1-\varepsilon}\big)+O_\varepsilon\big(\delta^{1+c_2-\varepsilon}\big).\nonumber
\end{eqnarray}

Let $S(t)$ be the above sum over the zeros. Its Fourier transform is
\[
\hat{S}(u)=\int_{-\infty}^{\infty}S(t)e(-tu)dt=\pi \sum_{|\gamma_F|\leq Z}a(\rho_F)X^{i\gamma_F}e(-\gamma_Fu)e^{-2\pi|u|}.
\]
By Plancherel's formula the integral in \eqref{101} equals
\begin{eqnarray*}
\frac{\pi}{2}\int_{-\infty}^{\infty}\bigg| \sum_{|\gamma_{F}|\leq Z} a(\rho_F)e^{i\gamma_F(Y+y)} \bigg|^{2}e^{-2|y|}dy,
\end{eqnarray*}
after the change of variables $Y=\log X$, $y=-2\pi u$. Hence
\begin{eqnarray*}
\int_{-\infty}^{\infty}\bigg| \sum_{|\gamma_{F}|\leq Z} a(\rho_F)e^{i\gamma_F(Y+y)} \bigg|^{2}e^{-2|y|}dy&=&\delta\Big(d_F\log\frac{1}{\delta}+\log\mathfrak{q}_F+(1-\gamma_0-\log2\pi)d_F\Big)\\
&&\qquad\qquad+O\big(\delta^{1+c}\big)+O_\varepsilon\big(\delta^{1+c_1-\varepsilon}\big)+O_\varepsilon\big(\delta^{1+c_2-\varepsilon}\big).\nonumber
\end{eqnarray*}
Lemma \ref{taub} leads to
\begin{eqnarray}\label{sumzeros}
\int_{X}^{2X}\bigg| \sum_{|\gamma_{F}|\leq Z} a(\rho_F)x^{\rho_F} \bigg|^{2}dx&=&\frac{3}{2}\delta X^2\Big(d_F \log\frac{1}{\delta}+\log\mathfrak{q}_F+(1-\gamma_0-\log2\pi)d_F\Big)\\
&&\qquad\quad+O\big(\delta^{1+c/2}X^2\big)+O_\varepsilon\big(\delta^{1+c_1/2-\varepsilon}X^2\big)+O_\varepsilon\big(\delta^{1+c_2/2-\varepsilon}X^2\big),\nonumber
\end{eqnarray}
after the change of variable $x=e^{Y+y}$, provided that
\begin{equation}\label{conditionX}
X^{1/A_2}\ll\delta^{-(1-c_1)}<\delta^{-(1+c_2)}\ll X^{1/A_1}.
\end{equation}

Next we use the explicit formula for $\psi_F(x)$ and get
\begin{eqnarray}\label{explicit}
\psi_F(x+\delta x)-\psi_F(x) - m_F\delta x&=&-\sum_{|\gamma_F|\leq Z}a(\rho_F)x^{\rho_F}+O\bigg((\log x)\min\Big\{1,\frac{x}{Z||x||}\Big\}\bigg)\\
&&\quad+O\bigg((\log x)\min\Big\{1,\frac{x}{Z||x+\delta x||}\Big\}\bigg)+O\big(xZ^{-1}(\log xZ)^2\big),\nonumber
\end{eqnarray}
where $||x||=\min_{n}|x-n|$ is the distance from $x$ to the nearest integer. Choosing $Z=X^2$ and using \eqref{sumzeros} we have
\begin{eqnarray*}
&&\int_{X}^{2X}\Big| \psi_F(x+\delta x)-\psi_F(x) - m_F\delta x\Big|^{2}dx=\frac{3}{2}\delta X^2\Big(d_F \log\frac{1}{\delta}+\log\mathfrak{q}_F+(1-\gamma_0-\log2\pi)d_F\Big)\\
&&\qquad\qquad\qquad\qquad+O\big(\delta^{1+c/2}X^2\big)+O_\varepsilon\big(\delta^{1+c_1/2-\varepsilon}X^2\big)+O_\varepsilon\big(\delta^{1+c_2/2-\varepsilon}X^2\big).\nonumber
\end{eqnarray*}
Summing over the dyadic intervals $[2^{-k}X,2^{-k+1}X]$, $1\leq k\leq K$, with 
\begin{equation}\label{conditionK}
2^K=\delta^{(1+c_2)A_1}X
\end{equation}
(so that \eqref{conditionX} still holds with $X$ being replaced by $2^{-K}X$) we obtain
\begin{eqnarray*}
&&\int_{2^{-K}X}^{X}\Big| \psi_F(x+\delta x)-\psi_F(x) - m_F\delta x\Big|^{2}dx\\
&&\qquad\qquad=\frac{(1-4^{-K})}{2}\delta X^2\Big(d_F \log\frac{1}{\delta}+\log\mathfrak{q}_F+(1-\gamma_0-\log2\pi)d_F\Big)\\
&&\qquad\qquad\qquad\qquad+O\big(\delta^{1+c/2}X^2\big)+O_\varepsilon\big(\delta^{1+c_1/2-\varepsilon}X^2\big)+O_\varepsilon\big(\delta^{1+c_2/2-\varepsilon}X^2\big).\nonumber
\end{eqnarray*}
For the integration in the range $[1,2^{-K}X]$ we use the first estimate of Lemma \ref{SVbound}. Hence
\begin{eqnarray*}
V_F(X,\delta)&=&\frac{1}{2}\delta X^2\Big(d_F \log\frac{1}{\delta}+\log\mathfrak{q}_F+(1-\gamma_0-\log2\pi)d_F\Big)+O\big(\delta^{1+c/2}X^2\big)\\
&&\qquad\qquad+O_\varepsilon\big(\delta^{1+c_1/2-\varepsilon}X^2\big)+O_\varepsilon\big(\delta^{1+c_2/2-\varepsilon}X^2\big)+O_\varepsilon\big(\delta^{1-\varepsilon}X^24^{-K}\big),\nonumber
\end{eqnarray*}
and then the theorem follows from \eqref{conditionX} and \eqref{conditionK}.

\subsection{Proof of Theorem B\ref{theoremB1}}

Integrating \eqref{B1condition} by parts we have
\begin{eqnarray*}
\int_{X}^{X_1}\Big |\psi_F(x+\delta x)-\psi_F(x) - m_F\delta x\Big|^{2}x^{-4} dx&=&\frac{1}{2}\delta X^{-2}\Big(d_F \log\frac{1}{\delta}+\log\mathfrak{q}_F+(1-\gamma_0-\log2\pi)d_F\Big)\\
&&\qquad\qquad+O\big(\delta^{1+c}X^{-2}\big)+O_\varepsilon\big(\delta^{1-\varepsilon}X_{1}^{-2}\big)
\end{eqnarray*}
uniformly for $\delta^{-1/B_2}\ll X,X_1\ll \delta^{-1/B_1}$. Similarly, the bound \eqref{SV1} leads to 
\begin{eqnarray*}
\int_{X_1}^{\infty}\Big |\psi_F(x+\delta x)-\psi_F(x) - m_F\delta x\Big|^{2}x^{-4} dx\ll_\varepsilon \delta^{1-\varepsilon}X_{1}^{-2}.
\end{eqnarray*}
Combining these estimates and \eqref{B1condition}, and letting $X_1= \delta^{-1/B_1}$ we get
\begin{eqnarray*}
&&\int_{0}^{\infty}\min\big\{x^2/X^{2},X^{2}/x^2\big\}\Big |\psi_F(x+\delta x)-\psi_F(x) - m_F\delta x\Big|^{2}x^{-2} dx\\
&&\qquad\qquad=\delta \Big(d_F \log\frac{1}{\delta}+\log\mathfrak{q}_F+(1-\gamma_0-\log2\pi)d_F\Big)+O\big(\delta^{1+c}\big)+O_\varepsilon\big(\delta^{1+2/B_1-\varepsilon}X^{2}\big).
\end{eqnarray*}

We now use the explicit formula \eqref{explicit} with $Z=X^2$. Writing $Y=\log X$ and $x=e^{Y+y}$ we obtain
\begin{eqnarray*}
\int_{-\infty}^{\infty}\bigg| \sum_{|\gamma_{F}|\leq Z} a(\rho_F)e^{i\gamma_F(Y+y)} \bigg|^{2}e^{-2|y|}dy&=&\delta \Big(d_F \log\frac{1}{\delta}+\log\mathfrak{q}_F+(1-\gamma_0-\log2\pi)d_F\Big)\\
&&\qquad\qquad+O\big(\delta^{1+c}\big)+O_\varepsilon\big(\delta^{1+2/B_1-\varepsilon}X^{2}\big).
\end{eqnarray*}
Retracing our steps as in the previous subsection leads to
\begin{eqnarray*}
&&\!\!\!\!\!\!\!\!\!\!\int_{-\infty}^{\infty}\bigg(\frac{\sin\kappa t}{t}\bigg)^2\bigg| \sum_{|\gamma_{F}|\leq Z} \frac{X^{i \gamma_{F}}}{1 + (t - \gamma_{F})^{2}} \bigg|^{2}dt\\
&&=\frac{\pi}{4} \kappa\Big(d_F\log\frac{1}{\kappa}+\log\mathfrak{q}_F+(1-\gamma_0-\log 4\pi)d_F\Big)+O\big(\kappa^{1+c}\big)+O_\varepsilon\big(\kappa^{1+2/B_1-\varepsilon}X^{2}\big).
\end{eqnarray*}
Lemma \ref{converselot} then implies that
\begin{eqnarray*}
&&\int_{-T}^{T}\bigg| \sum_{|\gamma_{F}|\leq Z} \frac{X^{i \gamma_{F}}}{1 + (t - \gamma_{F})^{2}} \bigg|^{2}dt\\
&&\qquad\qquad=\frac{T}{2}\Big(d_F\log T+\log \mathfrak{q}_F-(1+\log2\pi)d_F\Big)+O_\varepsilon\Big( T^{3/(3+c)+\varepsilon}\Big)\\
&&\qquad\qquad\qquad\qquad\quad+O_\varepsilon\big( T^{1-2c_1+\varepsilon}\big)+O_\varepsilon\big( T^{1-c_2/4+\varepsilon}\big)+O_\varepsilon\Big( \big(T^3X^2\big)^{B_1/(3B_1+2)+\varepsilon}\Big),
\end{eqnarray*}
provided that $T^{-(1+c_1)}\ll X^{-B_2}<X^{-B_1} \ll T^{-(1-c_2)}$. Moreover, we can restrict the summation over the zeros to $-T\leq\gamma_F,\gamma_F'\leq T$ and extend the range of the integration to $(-\infty,\infty)$ with an error term of size $\ll(\log T)^3$. Finally we choose $c_1$ and $c_2$ such that $T^{-(1+c_1)}= X^{-B_2}$ and $T^{-(1-c_2)}=X^{-B_1}$, and hence obtain the theorem.

\subsection{Proof of Theorem C\ref{secondtheorem}}

Consider the double integral 
\[
\int_{X}^{2X}\int_{H_1}^{H_2}\big|f(x,h)\big|^{2}dhdx,
\]
where $f(x,y)=\psi_F(x+y)-\psi_F(x) - m_Fy$ and $H_1<H_2<2H_1$. Here $H_1\asymp H_2\asymp H$ and $X^{1-B_3}\ll H\ll X^{1-B_1}$. Replacing $h$ by $\delta=h/x$ and changing the order of integration, this is equal to
\begin{eqnarray*}
&&\int_{H_1/2X}^{H_2/2X}\int_{H_1/\delta}^{2X}\big|f(x,\delta x)\big|^2xdxd\delta+\int_{H_2/2X}^{H_1/X}\int_{H_1/\delta}^{H_2/\delta}\big|f(x,\delta x)\big|^2xdxd\delta\\
&&\qquad\qquad+\int_{H_1/X}^{H_2/X}\int_{X}^{H_2/\delta}\big|f(x,\delta x)\big|^2xdxd\delta.
\end{eqnarray*}
By integration by parts, \eqref{2.1} implies that
\begin{eqnarray*}
\int_{X_1}^{X_2}\big|f(x,\delta x)\big|^2xdx=\frac{1}{3}\delta\big (X_{2}^{3}-X_{1}^{3}\big)\bigg(d_F \log\frac{1}{\delta}+\log \mathfrak{q}_F+(1-\gamma_0-\log2\pi)d_F\bigg)+O\big(\delta^{1+c} X^3\big),\nonumber
\end{eqnarray*}
provided that $X_1\asymp X_2\asymp X$. Hence
\begin{eqnarray}\label{601}
&&\int_{X}^{2X}\int_{H_1}^{H_2}\big|f(x,h)\big|^{2}dhdx=\frac{d_F}{2} X\bigg(H_{2}^{2} \log\frac{X}{H_2}-H_{1}^{2}\log\frac{X}{H_1}\bigg)\\
&&\qquad\quad+\frac{1}{4}\bigg(2\log \mathfrak{q}_F+\Big(1-2\gamma_0-2\log2\pi+4\log2\Big)d_F\bigg)X\big(H_{2}^{2}-H_{1}^{2}\big)+O\big(H^2X(H/X)^{c}\big)\nonumber
\end{eqnarray}
uniformly for 
\begin{equation}\label{600}
X^{1-B_3}\ll H\ll X^{1-B_1}.
\end{equation}

We now consider $X^{1-B_3}\ll H\ll X^{1-B_2}$. Summing \eqref{601} over the dyadic intervals $[2^{-k}X,2^{-k+1}X]$, $1\leq k\leq K$, with 
\[
K\asymp\frac{(1-B_1)\log X-\log H}{(1-B_1)\log 2}
\]
(so that \eqref{600} still holds with $X$ being replaced by $2^{-K}X$) we obtain
\begin{eqnarray*}
&&\int_{2^{-K}X}^{X}\int_{H_1}^{H_2}\big|f(x,h)\big|^{2}dhdx=\frac{(1-2^{-K})d_F}{2} X\bigg(H_{2}^{2} \log\frac{X}{H_2}-H_{1}^{2}\log\frac{X}{H_1}\bigg)\\
&&\qquad\qquad+\frac{(1-2^{-K})}{4}\bigg(2\log \mathfrak{q}_F+\Big(1-2\gamma_0-2\log2\pi+4\log2\Big)d_F\bigg)X\big(H_{2}^{2}-H_{1}^{2}\big)\\
&&\qquad\qquad\qquad\qquad-\frac{\big(2-2^{-K}(K+2)\big)(\log 2)d_F}{2}X\big(H_{2}^{2}-H_{1}^{2}\big)+O\big(H^2X(H/X)^{c}\big).
\end{eqnarray*}
Adding up the integration on $[1,2^{-K}X]$ using the second estimate of Lemma \ref{SVbound} we get
\begin{eqnarray}\label{200}
\int_{H_1}^{H_2}\int_{1}^{X}\big|f(x,h)\big|^{2}dxdh&=&\frac{d_F}{2} X\bigg(H_{2}^{2} \log\frac{X}{H_2}-H_{1}^{2}\log\frac{X}{H_1}\bigg)\\
&&\qquad\qquad+\frac{1}{4}\bigg(2\log \mathfrak{q}_F+\Big(1-2\gamma_0-2\log2\pi\Big)d_F\bigg)X\big(H_{2}^{2}-H_{1}^{2}\big)\nonumber\\
&&\qquad\qquad\qquad\qquad+O\big(H^2X(H/X)^{c}\big)+O_\varepsilon\Big(H^{2+1/(1-B_1)+\varepsilon}\Big).\nonumber
\end{eqnarray}

We now deduce \eqref{2.2} from \eqref{200}. In view of \eqref{200} we have
\begin{eqnarray*}
&&\int_{H}^{(1+\eta)H}\int_{1}^{X}\big|f(x,h)\big|^{2}dxdh=\eta H^2 X\bigg(d_F \log\frac{X}{H}+\log \mathfrak{q}_F-(\gamma_0+\log 2\pi)d_F\bigg)\nonumber\\
&&\qquad\qquad\qquad\qquad\qquad+O\bigg(\eta^2H^2X\log\frac{X}{H}\bigg)+O\big(H^2X(H/X)^{c}\big)+O_\varepsilon\Big(H^{2+1/(1-B_1)+\varepsilon}\Big).
\end{eqnarray*}
Let $g(x,h)=f(x,H)$. Since 
\[
\big|f\big|^2-\big|g\big|^2=2\big|f\big|\Big(\big|f\big|-\big|g\big|\Big)-\Big(\big|f\big|-\big|g\big|\Big)^2\ll \big|f\big|\big|f-g\big|+\big|f-g\big|^2 ,
\]
by Cauchy-Schwartz's inequality we get
\[
\int\int\Big(\big|f\big|^2-\big|g\big|^2\Big)\ll \bigg(\int\int \big|f\big|^2\bigg)^{1/2}\bigg(\int\int\big|f-g\big|^2\bigg)^{1/2}+\int\int\big|f-g\big|^2.
\]
As $f(x,h)-g(x,h)=f(x+H,h-H)$, using Lemma \ref{SVbound} we derive that
\begin{eqnarray*}
\int\int\big|f-g\big|^2&=&\int_{H}^{(1+\eta)H}\int_{1}^{X}\big|f(x+H,h-H)\big|^2dxdh\\
&=&\int_{0}^{\eta H}\int_{1+H}^{X+H}\big|f(x,h)\big|^2dxdh\\
&\ll&\eta^2H^2X\Big(\log\frac{X}{H}\Big)^2.
\end{eqnarray*}
Hence
\begin{eqnarray*}
&&\eta H\int_{1}^{X}\Big|\psi_F(x+H)-\psi_F(x)-m_FH\Big|^2dx= \int_{H}^{(1+\eta)H}\int_{1}^{X}\big|g(x,h)\big|^{2}dxdh\nonumber\\
&&\qquad\quad=\eta H^2 X\bigg(d_F \log\frac{X}{H}+\log \mathfrak{q}_F-(\gamma_0+\log 2\pi)d_F\bigg)\nonumber\\
&&\quad\quad\qquad\qquad+O\bigg(\eta^{3/2}H^2X\Big(\log\frac{X}{H}\Big)^{3/2}\bigg)+O\big(H^2X(H/X)^{c}\big)+O_\varepsilon\Big(H^{2+1/(1-B_1)+\varepsilon}\Big),
\end{eqnarray*}
and the theorem follows by choosing 
\[
\eta=\max\Big\{(H/X)^{2c/3},\big(HX^{-(1-B_1)}\big)^{2/3(1-B_1)}\Big\}. 
\]

\end{document}